\documentclass[reqno, 11pt]{article}

\usepackage{fullpage}

\usepackage[centertags]{amsmath}
\usepackage{amsmath}
\usepackage{amsfonts}
\usepackage{amssymb}
\usepackage{amsthm}
\usepackage{enumerate}
\usepackage{subfig}
\usepackage{tikz}

\usepackage{algorithm,algorithmic}
\usepackage{dsfont}

\usepackage{newlfont}
\usepackage{color}

\usepackage{authblk}

\usepackage{stmaryrd}
\SetSymbolFont{stmry}{bold}{U}{stmry}{m}{n}

\usepackage{stmaryrd}
\usepackage{mathrsfs}

\usepackage{fancyhdr}
\usepackage{graphicx}
\usepackage{fancybox}
\usepackage{setspace}
\usepackage{cleveref}

\newtheorem{thm}{Theorem}

\newtheorem{lem}[thm]{Lemma}
\newtheorem{prop}[thm]{Proposition}

\theoremstyle{definition}
\newtheorem{defn}[thm]{Definition}
\newtheorem{assump}[thm]{Assumption}

\theoremstyle{remark}
\newtheorem{rem}[thm]{Remark}
\newtheorem{nrem}[thm]{Notational Remark}

\newcommand{\R} {\mathbb{R}}

\newcommand{\E} {\mathbb{E}}

\newcommand{\p} {\mathbb{P}}

\newcommand{\indi}{\mathds{1}}

\newcommand{\caG}{{\mathcal G}}
\newcommand{\caN}{{\mathcal N}}

\newcommand{\caX}{{\mathcal X}}
\newcommand{\caY}{{\mathcal Y}}
\newcommand{\caL}{{\mathcal L}}
\newcommand{\caW}{{\mathcal W}}

\newcommand{\bsH}{{\boldsymbol H}}

\newcommand{\bsv}{{\boldsymbol v}}

\newcommand{\bsx}{{\boldsymbol x}}
\newcommand{\bsy}{{\boldsymbol y}}
\newcommand{\bsz}{{\boldsymbol z}}

\newcommand{\bss}{{\boldsymbol \sigma}}

\newcommand{\wt}{\widetilde}
\newcommand{\ol}{\overline}

\newcommand{\beq}{ \begin{equation} }
\newcommand{\eeq}{ \end{equation} }

\newcommand{\dd}{\mathrm{d}}

\newcommand{\SNR}{\omega}
\newcommand{\cn}{\tau}
\newcommand{\F}{\cn_1}
\newcommand{\G}{\cn_2}

\makeatletter
\def\blfootnote{\xdef\@thefnmark{}\@footnotetext}
\makeatother

\numberwithin{equation}{section} 
\numberwithin{thm}{section}

\title{Universality of the fluctuations of the free energy in generalized Sherrington--Kirkpatrick models and the log likelihood ratio in spiked Wigner models}

\author{Hyunsuk Choo\footnote{Department of Mathematical Sciences, KAIST, Daejeon, 34141, Korea \newline email: \texttt{hyunsuk.choo@kaist.ac.kr}},
Yoochan Han\footnote{Department of Mathematical Sciences, KAIST, Daejeon, 34141, Korea \newline email: \texttt{yoochan.han@kaist.ac.kr}},
and Ji Oon Lee\footnote{Department of Mathematical Sciences, KAIST, Daejeon, 34141, Korea \newline email: \texttt{jioon.lee@kaist.edu}}}

\begin{document}

\maketitle

\begin{abstract}
We consider the fluctuations of the free energy in generalized Sherrington--Kirkpatrick models and the log likelihood ratio of spiked Wigner models in the high temperature/subcritical regime. We prove that the limiting laws of the fluctuations are Gaussian under suitable assumptions, and the result is universal in the sense that it does not depend on the distribution of the disorder or the prior except that the means and the variances of the limiting laws depend on a few parameters of the model. The proof is based on the multigraph expansion that provides a unified approach to analyze both models.
\end{abstract}

\section{Introduction}

The theoretical study of the spin glass was initiated about half a century ago, where the main motivation of the study was to understand the behavior of the magnetic properties of substitutional alloys \cite{sherrington202550}. Since then, it has generated plenty of literature in various fields of study, including statistical physics, probability theory, and computer science. 

The Sherrington--Kirkpatrick (SK) model \cite{sherrington1975solvable} is one of the most canonical and well-studied models in the theory of spin glass. The SK model is defined by the Hamiltonian
\[
	H_N(\bss) = \frac{1}{\sqrt{N}} \sum_{i<j} J_{ij} \sigma_i \sigma_j,
\]
where the disorder $J_{ij}$'s ($1 \leq i < j \leq N$) are independent standard Gaussian random variables and the spin vector $\bss = (\sigma_1, \dots, \sigma_N) \in \{-1, 1 \}^N$. The most fundamental objects of the study are the partition function and the free energy, defined by
\beq \label{eq:partition_function}
	Z_N(\beta) := \frac{1}{2^N} \sum_{\bss \in \{-1, 1 \}^N} \exp (\beta H_N(\bss)), \quad F_N(\beta) := \frac{1}{N} \log Z_N(\beta),
\eeq
where the parameter $\beta > 0$ is the inverse temperature. The deterministic limit $F(\beta)$ of the free energy $F_N(\beta)$ was first predicted by Parisi \cite{parisi1979infinite,parisi1980sequence} and rigorously proved by Talagrand \cite{talagrand2006parisi}. The limit is universal in the sense that it does not depend on the detail of the interaction $J_{ij}$ under suitable assumptions on the moments of the disorder \cite{guerra2002thermodynamic,carmona2006universality}.

The SK model can be generalized by changing the distribution of the spins (called the prior) from the Ising type (or Rademacher prior) to an arbitrary probability measure, while maintaining the mean-field nature. One of the most fundamental generalizations of the SK model is the spherical Sherrington--Kirkpatrick (SSK) model, introduced by Kosterlitz, Thouless, and Jones \cite{kosterlitz1976spherical}, where the prior is given as the uniform measure on the (hyper)sphere of radius $\sqrt{N}$ in $\R^N$ (called the spherical prior). The free energy of the SSK model was computed in \cite{kosterlitz1976spherical,crisanti1992spherical} and later rigorously proved in \cite{talagrand2006free}. There are also numerous works on the free energy of generalized SK models with different priors; see, e.g., \cite{arous2001aging,panchenko2005free,panchenko2018free}.

The second order term, or the fluctuation, of the free energy has also been studied. In the seminal work by Aizenman, Lebowitz, and Ruelle \cite{AizenmanLebowitzRuelle}, it was proved that the (rescaled) fluctuation of the free energy in the SK model, $N(F_N(\beta) - F(\beta))$, is asymptotically Gaussian in the high temperature regime ($\beta < 1$). Their result is universal under suitable assumptions on the disorder, except that the mean and the variance of the fluctuation depend on the fourth moment of the disorder. Other methods have also been used to prove the asymptotic normality of the free energy of the SK model \cite{frohlich1987some,comets1995sherrington,talagrand2010mean2}. Analogous results were proved for the diluted SK model \cite{kosters2006fluctuations}, the SK model with weak external field \cite{guerra2002central,chen2017fluctuations,dey2023mean}, and the SK model with the Curie--Weiss interaction \cite{banerjee2020fluctuation}. For more results about the free energy fluctuation, we refer to \cite{collins2025free} and the references therein.

While the free energy fluctuation is an interesting object per se, it is also of fundamental importance due to its natural connection to the detection problem in spiked Wigner model. A spiked Wigner matrix $M$ is of the form
\beq \label{eq:rank-1}
	M = W + \sqrt{\lambda} \bsx \bsx^T,
\eeq
where the signal (or the spike) $\bsx$ is a unit vector in $\R^N$ and the noise $W$ is a Wigner matrix. (See Definition \ref{defn:Wigner}.) The parameter $\lambda$ denotes the strength of the signal, known as the signal-to-noise ratio (SNR). The detection means to test whether the given data matrix is drawn from a spiked distribution $(\lambda > 0)$ or an unspiked distribution $(\lambda=0)$. More precisely, it is the hypothesis test between $\bsH_0 : \lambda = 0$ and $\bsH_1 : \lambda = \SNR > 0$ for some pre-determined parameter $\SNR$. A natural object we can use for the detection is the Radon-Nikodym derivative involving the unspiked distribution (under $\bsH_0$) and the spiked distribution (under $\bsH_1$), which is known as the likelihood ratio in statistics. It is well-known that the tests based on the likelihood ratio are optimal in terms of the sum of Type-I and Type-II errors (or the tradeoff between the errors), due to Neyman--Pearson lemma \cite{neyman1933ix}. Thus, the optimal error of the detection can be computed from the fluctuation of the likelihood ratio.

In the simplest case where the Wigner matrix $W$ is a GOE matrix, the likelihood ratio given $\bsH_0$ reduces to
\[
	\caL (\SNR) = \E_\caX \left[ e^{-\SNR N\| \bsx \|^4 /4} \prod_{i<j} \exp \left( \sqrt{\SNR} N W_{ij} x_i x_j \right) \prod_{k=1}^N \exp \left( \frac{\sqrt{\SNR}}{2} N W_{kk} x_k^2 \right) \right].
\]
(See Definitions \ref{defn:LR_normal} and \ref{defn:LR_iid} for the definition of the likelihood ratio when $W$ is non-Gaussian.) When the prior is Rademacher, i.e., $x_i$'s are i.i.d. with $\p(\sqrt{N} x_i = 1) = \p(\sqrt{N} x_i = -1) = 1/2$, we can see that the log likelihood ratio
\[
	\log \caL (\SNR) = -\frac{\SNR N}{4} + \sum_{k=1}^N \frac{\sqrt{\SNR}}{2} W_{kk} + \log \left( \frac{1}{2^N} \sum_{\sqrt{N} \bsx \in \{-1, 1 \}^N} \exp \left( \sqrt{\SNR} N \sum_{i<j} W_{ij} x_i x_j \right) \right),
\]
and if $\SNR < 1$, its limiting fluctuation can be readily obtained from the fluctuation of the free energy in the SK model after setting $\beta = \sqrt{\SNR}$, $J_{ij} = \sqrt{N} W_{ij}$, and $\sigma_i = \sqrt{N} x_i$. A similar result for more general priors was proved by El Alaoui, Krzakala, and Jordan \cite{AlaouiJordan2018} using the approaches based on the study of the SK model. For more discussion on the results and techniques that were developed for the SK model and applied to the spiked Wigner model with the Gaussian disorder, we refer to \cite{zdeborova2016statistical,bandeira2018notes} and references therein.

The connection between the free energy and the log likelihood ratio is lost when the disorder (or noise) is non-Gaussian, and not much is known about the fluctuation of the free energy or the log likelihood ratio in this case, except for the SSK model. In the SSK model, due to the rotational symmetry, the spectral properties of the disorder completely determine the free energy via the integral representation formula \cite{Baik-Lee2016}, which can be analyzed in detail by applying the machinery of random matrix theory. It is notable that the variances of the fluctuations of the free energy in the SK model and the SSK model coincide, whereas the means do not. (See Remark \ref{rem:comparison}.) Even less is known about the fluctuation of the log likelihood ratio; to our best knowledge, the only known results are for the (two block) stochastic block model \cite{banerjee2018contiguity} and for the model with Rademacher prior \cite{chung2025asymptotic}. In particular, the fluctuation of the log likelihood ratio with the spherical prior, which is one of the most natural priors in the detection problem, could not be proved, even with the random matrix theory.

Nevertheless, when comparing the known results, we can find striking similarity among them. For example, the mean and the variance of the limiting Gaussian distribution for the free energy in SSK model are
\[
	m = \frac{1}{4} \log (1- \beta^2) - \frac{\beta^2}{2} + \frac{\beta^4 (w_4 -3)}{8}, \quad V = -\frac{1}{2} \left( \log (1- \beta^2) + \beta^2 - \frac{\beta^4 (w_4 -3)}{4} \right),
\]
where $w_4$ is the fourth moment of the disorder, and those for the log likelihood ratio with Rademacher prior are, if we ignore the diagonal entries,
\[
	m= \frac{1}{4} \log (1-\SNR F_p) + \frac{\SNR F_p}{4} + \frac{\SNR^2 (2F_p^2 - G_p)}{16}, \quad V= -\frac{1}{2} \left( \log (1-\SNR F_p) + \SNR F_p + \frac{\SNR^2}{4} (2F_p^2 - G_p) \right)
\]
where $F_p$ and $G_p$ are some constants that depend only on the density $p$. (See Theorem \ref{thm:log_LR} for the precise definition of $F_p$ and $G_p$.) This suggests that there is certain inter-model universality in the level of the fluctuation in the high temperature regime except a few model-specific parameters such as $w_4$, $F_p$, and $G_p$.

\subsection{Main results}

In this paper, we prove the limit theorems for the fluctuations of the free energy and the log likelihood ratio for large classes of disorders and priors in the high temperature regime. We first state the result on the free energy. Recall that a real-valued random variable $z$ is strictly sub-Gaussian if $\E[e^{\alpha z}] \leq e^{\alpha^2 /2}$ for any (possibly $N$-dependent) constant $\alpha$.

\begin{thm} \label{thm:free_energy_1}
Suppose that $\sqrt{N} W_{ij}$ $(1 \leq i<j \leq N)$ are i.i.d., symmetric, and strictly sub-Gaussian random variables with variance $1$. Suppose that $\sqrt{N} x_k$ $(1 \leq k \leq N)$ are i.i.d., symmetric, and strictly sub-Gaussian random variables with variance $1$, independent of $\sqrt{N} W_{ij}$. Let
\beq \label{eq:free_energy_def}
	Z_N \equiv Z_N(\beta) = \E_{\caX} \left[ \exp \left( \beta N \sum_{i<j} W_{ij} \frac{x_i x_j}{\| \bsx \|^2} \right) \right], \qquad F_N \equiv F_N(\beta) = \frac{1}{N} \log Z_N.
\eeq
(When $\| \bsx \| = 0$, we define $x_i x_j / \| \bsx \|^2 = 0$ for all $i, j$.) Then, for any $\beta \in (0, 1)$,
\[
	F_N(\beta) \to F(\beta) := \beta^2 /4
\]
and
\[
	N(F_N(\beta) - F(\beta) ) \Rightarrow \caN(m_F, V_F),
\]
as $N \to \infty$, i.e., $N(F_N(\beta) - F(\beta) )$ converges in distribution to the Gaussian with mean $m_F$ and the variance $V_F$ as $N \to \infty$. Here,
\[
	m_F = \frac{1}{4} \log (1- \beta^2) + \frac{\beta^2 (1-m_4)}{4} + \frac{\beta^4 (w_4-3)(m_4^2 -3)}{48}
\]
and
\[
	V_F = -\frac{1}{2} \left( \log (1- \beta^2) + \beta^2 - \frac{\beta^4 (w_4-3)}{4} \right),
\]
where we let $w_4 = N^2 \E[W_{ij}^4]$ and $m_4 = N^2 \E [x_1^4]$.
\end{thm}

Theorem \ref{thm:free_energy_1} establishes the universal limiting behavior of the free energy fluctuation in the high temperature regime ($\beta < 1$) except that (i) the mean and the variance depend on the rescaled fourth moment of the disorder (or the fourth cumulant $\kappa_4 : =w_4 -3$) and (ii) the mean depends on the rescaled fourth moment of the prior. (See also Remark \ref{rem:comparison}.) The threshold is optimal, since the mean and the variance in Theorem \ref{thm:free_energy_1} diverges as $\beta \nearrow 1$. In the low temperature regime ($\beta > 1$), the limiting behavior of the free energy fluctuation is believed to be completely different from that in the high temperature regime. (We refer to \cite{Baik-Lee2016} and references therein for more discussion.) We also remark that Theorem \ref{thm:free_energy_1} implies the universality of the limiting free energy in the high temperature regime, which is given by $\beta^2/4$. 

Note that the spins are strictly normalized in the sense that the spin vector $\bsx/\| \bsx \|$ is always a unit vector; see the discussion below Theorem \ref{thm:free_energy} for more information. As special cases, if $\sqrt{N} x_k$'s are Rademacher, our model reduces to the SK model with general disorder. Also, if $\sqrt{N} x_k$'s are standard Gaussian, we recover the SSK model with general disorder. Theorem \ref{thm:free_energy_1} also holds even when the disorder and the prior are not strictly sub-Gaussian under suitable assumptions; the results for more general cases will be stated in Section \ref{subsec:spin}.

Next, we state the result on the log likelihood ratio. We will assume the existence of the densities of the disorder and the following regularity conditions for them. (If the distribution of the disorder has an atom, the detection problem becomes uninteresting; see, e.g., Appendix G in \cite{Perry2018}.)

\begin{assump} \label{assump:decay}
Suppose that $\sqrt{N} W_{ij}$ $(1 \leq i<j \leq N)$ are i.i.d. with the density function $p$, and have variance $1$. Similarly, suppose that the normalized diagonal entries $\sqrt{N} W_{kk}$ $(1 \leq k \leq N)$ are i.i.d. with the density function $p_d$, and have variance $w_2$ for a constant $w_2 \geq 0$. We assume the following:
	\begin{itemize}
	\item The density functions $p$ and $p_d$ are smooth, positive everywhere, and symmetric (about $0$).
	\item The functions $p$, $p_d$, and all their derivatives vanish at infinity.
	\item The functions $p^{(s)}/p$ and $p_d^{(s)}/p_d$ are polynomially bounded, i.e., for any positive integer $s$ there exist constants $C_s, r_s >0$, independent of $N$, such that $|p^{(s)}(x)/p(x)|, |p_d^{(s)}(x)/p_d(x)| \leq C_s (1+|x|)^{r_s}$ uniformly on $x$. (Here, $p^{(s)}$ and $p_d^{(s)}$ are the $s$-th derivatives of $p$ and $p_d$, respectively.)
	\end{itemize}
\end{assump}

\begin{thm} \label{thm:log_LR_1}
Suppose that Assumption \ref{assump:decay} holds for $W$. Suppose that $\sqrt{N} x_k$  $(1 \leq k \leq N)$ are i.i.d., symmetric, and strictly sub-Gaussian random variables with variance $1$, independent of $W$. Define the likelihood ratios by
\beq
	 \caL(\SNR) := \E_\caX \left[ \prod_{i<j} \frac{p(\sqrt{N} W_{ij} - \sqrt{\SNR N} x_i x_j /\| \bsx \|^2)}{p(\sqrt{N} W_{ij})} \prod_k \frac{p_d(\sqrt{N} W_{kk} - \sqrt{\SNR N} x_k^2  /\| \bsx \|^2)}{p_d(\sqrt{N} W_{kk})} \right]
\eeq
and
\beq
	\ol \caL (\SNR) := \E_\caX \left[ \prod_{i<j} \frac{p(\sqrt{N} W_{ij} - \sqrt{\SNR N} x_i x_j)}{p(\sqrt{N} W_{ij})} \prod_k \frac{p_d(\sqrt{N} W_{kk} - \sqrt{\SNR N} x_k^2)}{p_d(\sqrt{N} W_{kk})} \right].
\eeq
Define
\[
	F_p:= \int_{-\infty}^{\infty} \frac{(p'(x))^2}{p(x)} \dd x, \quad F_d:= \int_{-\infty}^{\infty} \frac{(p_d'(x))^2}{p_d(x)} \dd x, \quad G_p:= \int_{-\infty}^{\infty} \frac{(p''(x))^2}{p(x)} \dd x.
\]
Then, for any $\SNR \in (0, 1/F_p)$,
\[
	\log \caL(\SNR) \Rightarrow \caN(-\rho_L, 2\rho_L), \quad \log \ol \caL(\SNR) \Rightarrow \caN(-\rho_L, 2\rho_L),
\]
as $N \to \infty$, i.e., $\log \caL(\SNR)$ and $\log \ol\caL(\SNR)$ converge in distribution to the Gaussian with mean $-\rho_L$ and the variance $2\rho_L$ as $N \to \infty$, where
\[
	\rho_L:= -\frac{1}{4} \left( \log (1-\SNR F_p) + \SNR(F_p- 2F_d) + \frac{\SNR^2}{4} (2F_p^2 - G_p) \right).
\]
\end{thm}

See Section \ref{subsec:signal} for more discussion on the definition of the likelihood ratios in Theorem \ref{thm:log_LR_1}.

Theorem \ref{thm:log_LR_1} proves the universality of the log likelihood ratio in the subcritical regime ($\SNR <1/F_p$) except that the mean and the variance depend on $F_p$, $F_d$, and $G_p$. In particular, we find that the limiting distribution of the log likelihood ratio with the spherical prior is the Gaussian in the subcritical regime with mean $-\rho_L$ and the variance $2\rho_L$, by considering $\log \caL(\SNR)$ in the case $\sqrt{N} x_k$'s are standard Gaussian. Like in Theorem \ref{thm:free_energy_1}, the threshold is optimal, since it is known that the strong detection, a consistent test between $\bsH_0$ and $\bsH_1$ with probability $1-o(1)$, is possible in the supercritical regime ($\SNR > 1/F_p$). (We refer to \cite{Perry2018} and references therein for more discussion.)

The parameters $F_p$ and $F_d$ are the Fisher information of the off-diagonal entries and the diagonal entries of the data matrix, respectively, and their appearance in the limiting distribution is natural when considering the optimality of the pre-transformed PCA \cite{lesieur2015mmse,Perry2018}. On the other hand, the parameter $G_p$ can only be seen when we precisely compute the limiting law of the log likelihood ratio. To our best knowledge, it first appeared in \cite{chung2025asymptotic}, and contrary to the conjecture made in \cite{chung2025asymptotic}, its appearance is not due to the Rademacher prior. 

The corresponding results for more general cases will be stated in Section \ref{subsec:signal}.

\subsection{Main contributions}

To explain our main contribution, we need to separate the universality of the fluctuation into two different types. The first is the disorder universality, which asserts that for a specific prior, the fluctuation does not depend on the disorder. The second is the prior universality, which asserts that for a specific disorder, the fluctuation does not depend on the prior. In this viewpoint, the results for the SK model \cite{AizenmanLebowitzRuelle} or the SSK model \cite{Baik-Lee2016} prove the disorder universality and the result in \cite{AlaouiJordan2018} proves the prior universality in the log likelihood ratio with Gaussian disorder. However, the methods in these works are model-specific and rely on certain special properties of the prior - Ising type spin in the cluster expansion in \cite{AizenmanLebowitzRuelle} or the spherical symmetry in the integral representation in \cite{Baik-Lee2016} - or of the disorder - Gaussianity in the cavity method \cite{AlaouiJordan2018}.

To prove both the disorder universality and the prior universality, we introduce a novel method based on the graph expansion, which provides a general, unified approach to analyze both the free energy and the log likelihood ratio. In this method, the analysis based on the special properties of the Gaussianity of the disorder or explicit formulas from the Rademacher prior or the spherical prior are entirely removed. Instead, the object is first Taylor-expanded and written as the partial expectation of a polynomial of the spin (or spike) variables. (See Proposition \ref{prop:main}.) Then, we convert each term in the polynomial into a (multi-)graph with a certain structure where the disorder corresponds to the edge and the prior to the node. (See Section \ref{sec:multigraph}.) By estimating the $L^2$-norm difference between the graphs from the given prior and the matching graphs from the Rademacher prior, we can prove the prior universality. Finally, we can prove the disorder universality with the Rademacher prior, adapting the idea of the cluster expansion.

The main difficulty in the analysis based on the graph expansion with a general prior is that the control of the large graph, whose size increases with $N$, is almost impossible if the (fourth and higher) moments of the prior distribution is extremely large. In fact, this is not just a technical issue but a fundamental one due to the nature of the generalized models. Even with the Gaussian disorder, for certain priors with sufficiently large moments, the prior universality for the log likelihood ratio holds only when $\SNR < \lambda_c$ for some $\lambda_c < 1$, known as the reconstruction threshold \cite{AlaouiJordan2018}. This suggests that we need to assume a certain condition that jointly depends on the prior and the inverse temperature $\beta$ (or the SNR $\SNR$), which preferably does not depend on the disorder. In this paper, we use Assumption \ref{assump:second_moment} for this purpose. (See Section \ref{subsec:threshold} for more detail.)

We expect our method can be applied to a very general class of models. While we assume several assumptions on the disorder and the prior in this paper, most of them are technical ones and we believe that they can be relaxed further by more refined analysis of the graphs. However, we will mainly focus on the introduction of our method and the proof of the main results based on it; the analysis under weaker assumptions will be discussed in future works.

\subsection{Organization of the paper}
The rest of the paper is organized as follows. In Section \ref{sec:main}, we state the general results, Theorem \ref{thm:free_energy} and Theorem \ref{thm:log_LR}. In Section \ref{sec:thm_proof}, we prove the results by first justifying the Taylor expansion of the partition function and the likelihood ratio, and then applying Proposition \ref{prop:main} that asserts the convergence of the partial expectation of the polynomial of the spin (or spike) variables. In Sections \ref{sec:multigraph} and \ref{sec:main_proof}, we introduce the graph expansion method and prove Proposition \ref{prop:main}. Several technical results are proved in Section \ref{sec:proofs}.

\section{General results} \label{sec:main}

In this section, we state our main results, Theorem \ref{thm:free_energy} and Theorem \ref{thm:log_LR}, which generalize Theorem \ref{thm:free_energy_1} and Theorem \ref{thm:log_LR_1}, respectively. We assume that the disorder in both the free energy and the log likelihood ratio is given as a Wigner matrix for which we use the following definition.

\begin{defn}[Wigner matrix] \label{defn:Wigner}
A Wigner matrix is an $N \times N$ real symmetric matrix $W$ whose upper triangle entries $W_{ij}$ are i.i.d. real random variables satisfying the following conditions:
\begin{itemize}
\item The entries are centered, i.e., $\E[W_{ij}] = 0$ for all $i, j$.
\item The variances of the entries are normalized so that $\E[W_{ij}^2] = N^{-1}$ for $i \neq j$ and $\E[W_{kk}^2] = w_2 N^{-1}$ for a constant $w_2 \geq 0$.
\item For any integer $s>2$, there exists a constant $C^W_s$ such that $\E[|\sqrt{N} W_{ij}|^s] \leq C^W_s$ for all $i, j$.
\end{itemize}
\end{defn}

Similarly, we use the following definition for the prior in both the free energy and the log likelihood ratio.

\begin{defn}[Prior] \label{defn:prior}
For a random variable $\bsx = (x_1, \dots, x_N) \in \R^N$ whose distribution is given by a prior $\caX$, we assume that for $i=1, 2, \dots, N$, 
\begin{itemize}
\item The entries $x_i$'s are i.i.d. and the distribution of $\sqrt{N} x_i$ does not depend on $N$.
\item The entries $x_i$'s are symmetric (about $0$) and their variances are normalized so that $\E[x_i^2] = N^{-1}$.
\item for any integer $s>2$, there exists a constant $C^x_s$ such that $\E[|\sqrt{N} x_i|^s] \leq C^x_s$.
\end{itemize}
\end{defn}

In the statement and the proof of our main results, we use the following assumption.

\begin{assump} \label{assump:second_moment}
For some constant $\SNR_c > 0$, there exists a sequence of events $\Omega_N$ in $\caX$ with $\p(\Omega_N^c) = o(1)$ such that
\beq \label{eq:second_moment_formula}
	\E_{\bsx, \bsx' \sim \caX} \left[ \indi(\Omega_N (\bsx)) \indi(\Omega_N (\bsx')) \exp \left( \frac{N \SNR_c \langle \bsx, \bsx' \rangle^2}{2} \right) \right] = O(1).
\eeq
\end{assump}

The notation $\bsx, \bsx' \sim \caX$ means that the spin vectors $\bsx$ and $\bsx'$ are independently drawn from the same prior distribution $\caX$. An important example of priors that satisfying Assumption \ref{assump:second_moment} is a sub-Gaussian random variable for which we use the following definition:

\begin{defn} \label{defn:sub_Gaussian}
A real valued random variable $z$ is sub-Gaussian with the variance proxy $\sigma^2$ if $\E[z] = 0$ and $\E[e^{\alpha z}] \leq e^{\sigma^2 \alpha^2 /2}$ for any (possibly $N$-dependent) constant $\alpha$.
\end{defn}
For other equivalent definitions and properties of the sub-Gaussian random variables, we refer to Section 2.5 in \cite{vershynin2020high}. It is known that if the prior is sub-Gaussian with variance proxy $\sigma^2$, then Assumption \ref{assump:second_moment} holds with any $\SNR_c < 1/\sigma^2$. Since a strictly sub-Gaussian random variable is sub-Gaussian with the variance proxy $1$, Assumption \ref{assump:second_moment} holds with any $\SNR_c < 1$ if the prior is strictly sub-Gaussian.

In Section \ref{subsec:threshold} we will discuss Assumption \ref{assump:second_moment} more in detail.

\subsection{Free energy of spin glass} \label{subsec:spin}

Recall the definition of the partition function and the free energy of a generalized SK model.

\begin{defn}[Partition function and free energy] \label{defn:partition_function}
For an $N \times N$ Wigner matrix $W = (W_{ij})$ and a prior $\caX$, we define the partition function $Z_N(\beta)$ and the free energy $F_N(\beta)$ at inverse temperature $\beta > 0$ by
\[
	Z_N \equiv Z_N(\beta) = \E_{\caX} \left[ \exp \left( \beta N \sum_{i<j} W_{ij} \frac{x_i x_j}{\| \bsx \|^2} \right) \right], \qquad F_N \equiv F_N(\beta) = \frac{1}{N} \log Z_N.
\]
Here, when $\| \bsx \| = 0$, we define $x_i x_j / \| \bsx \|^2 = 0$ for all $i, j$.
\end{defn}

Our result on the asymptotic normality of the free energy, which generalizes Theorem \ref{thm:free_energy_1}, is as follows:

\begin{thm} \label{thm:free_energy}
Recall the definition of $F_N(\beta)$ in Definition \ref{defn:sub_Gaussian}. Suppose that $W$ is a Wigner matrix, $\sqrt{N} W_{ij}$ $(i<j)$ is symmetric, and one of the following is true:
\begin{itemize}
\item the (rescaled) entry $\sqrt{N} W_{ij}$ is strictly sub-Gaussian, or
\item the (rescaled) prior $\sqrt{N} x_i$ is bounded, and its distribution is absolutely continuous in a neighborhood of $0$.
\end{itemize}
Suppose that Assumption \ref{assump:second_moment} holds with $\SNR_c$. Then, the statement in Theorem \ref{thm:free_energy_1} holds without any change for any $\beta \in (0, \sqrt{\SNR_c})$. In particular, it holds for any $\beta \in (0, 1/\sigma)$ if the prior is sub-Gaussian with variance proxy $\sigma^2$.
\end{thm}

We prove Theorem \ref{thm:free_energy} in Section \ref{subsec:free_energy}.

If we use the i.i.d. spins instead of the normalized spins so that the partition function is defined by $Z_N(\beta) = \E_{\caX}[\exp( \beta N \sum W_{ij} x_i x_j)]$, then heuristically the limiting free energy is replaced by $\beta^2 \| \bsx \|^4 /4$, which would generate a sub-leading order term due to the order $N^{-1/2}$ fluctuation of $\| \bsx \|$, which may dominate the fluctuation of the free energy described in Theorem \ref{thm:free_energy}. To avoid such complication, we only consider the normalized spins in this paper.

The assumptions in Theorem \ref{thm:free_energy} about the disorder or the prior are used in our proof to justify the first step of our method, the Taylor expansion of the exponential function in the definition of the partition function. In principle, these can be replaced by any other assumptions that guarantee the validity of the Taylor expansion.

\begin{rem} \label{rem:comparison}
For Rademacher prior, $m_4=1$ and hence
\[
	m_F = \frac{1}{4} \log (1- \beta^2) - \frac{\beta^4 \kappa_4}{24}, \quad V_F = -\frac{1}{2} \left( \log (1- \beta^2) + \beta^2 - \frac{\beta^4 \kappa_4}{4} \right),
\]
which recovers the result in \cite{AizenmanLebowitzRuelle}. Similarly, for the spherical prior, which is equal to the normalized prior with Gaussian entries, we let $m_4 = 3$ to find that
\[
	m_F = \frac{1}{4} \log (1- \beta^2) - \frac{\beta^2}{2} + \frac{\beta^4 \kappa_4}{8}, \quad V_F = -\frac{1}{2} \left( \log (1- \beta^2) + \beta^2 - \frac{\beta^4 \kappa_4}{4} \right).
\]
It recovers the result in \cite{Baik-Lee2016} when $\sqrt{N} W_{ij}$ is strictly sub-Gaussian. (Note that the inverse temperature $\beta$ in \cite{Baik-Lee2016} is effectively the half of that in the current paper, since the Hamiltonian in \cite{Baik-Lee2016} is defined using the sum $\sum_{i \neq j}$ instead of $\sum_{i < j}$.)
\end{rem}

\subsection{Log likelihood ratio} \label{subsec:signal}

Suppose that for a given spiked Wigner matrix of the form \eqref{eq:rank-1}, the true spike $\bsx$ is not known a priori, but only its prior distribution $\caX$ is known. To prove the fluctuation of the log likelihood ratio, it suffices to prove it under one hypothesis as the other one can be readily obtained by applying Le Cam's third lemma. (See, e.g., Theorem 6.6 and Example 6.7 in \cite{van2000asymptotic}.) We assume $\bsH_0$, i.e., $M = W$ in \eqref{eq:rank-1}, since the likelihood ratio under $\bsH_0$ naturally corresponds to the free energy of the generalized SK model in our analysis. The likelihood ratio for normalized priors is defined as follows:

\begin{defn}[Likelihood ratio for normalized prior] \label{defn:LR_normal}
For an $N \times N$ Wigner matrix $W = (W_{ij})$ and a prior $\caX$, we define the likelihood ratio given $\bsH_0$, $\caL(\SNR)$, by
\beq
	 \caL(\SNR) = \E_\caX \left[ \prod_{i<j} \frac{p(\sqrt{N} W_{ij} - \sqrt{\SNR N} x_i x_j /\| \bsx \|^2)}{p(\sqrt{N} W_{ij})} \prod_k \frac{p_d(\sqrt{N} W_{kk} - \sqrt{\SNR N} x_k^2  /\| \bsx \|^2)}{p_d(\sqrt{N} W_{kk})} \right].
\eeq
Here, when $\| \bsx \| = 0$, we define $x_i x_j / \| \bsx \|^2 = 0$ for all $i, j$.
\end{defn}

As we will see in Theorem \ref{thm:log_LR}, for the likelihood ratio, there are no terms corresponding to the limiting free energy but the leading order term is asymptotically normally distributed. In particular, unlike the free energy of the spin glass model, the difference between the i.i.d. prior and the normalized prior is negligible. We thus consider the likelihood ratio for i.i.d. priors as follows:

\begin{defn}[Likelihood ratio for i.i.d. prior] \label{defn:LR_iid}
For an $N \times N$ Wigner matrix $W = (W_{ij})$ and a prior $\caX$, we define the likelihood ratio $\ol \caL (\SNR)$ by
\beq \label{eq:LR_iid}
	\ol \caL (\SNR) = \E_\caX \left[ \prod_{i<j} \frac{p(\sqrt{N} W_{ij} - \sqrt{\SNR N} x_i x_j)}{p(\sqrt{N} W_{ij})} \prod_k \frac{p_d(\sqrt{N} W_{kk} - \sqrt{\SNR N} x_k^2)}{p_d(\sqrt{N} W_{kk})} \right].
\eeq
\end{defn}

Our result on the asymptotic normality of the log likelihood ratio, which generalizes Theorem \ref{thm:log_LR_1}, is as follows:

\begin{thm} \label{thm:log_LR}
Recall the definition of the likelihood ratio $\caL(\SNR)$ in Definition \ref{defn:LR_normal} and $\ol \caL(\SNR)$ in Definition \ref{defn:LR_iid}. Suppose that $W$ satisfies Assumption \ref{assump:decay}. Suppose that Assumption \ref{assump:second_moment} holds with $\SNR_c$. Then, the statement in Theorem \ref{thm:log_LR_1} holds without any change for any $\SNR \in (0, \SNR_c/F_p)$. In particular, it holds for any $\SNR \in (0, 1/(F_p \sigma^2))$ if the prior is sub-Gaussian with variance proxy $\sigma^2$.
\end{thm}

We prove Theorem \ref{thm:log_LR} in Section \ref{subsec:log_LR}.

\begin{rem}
Theorem \ref{thm:log_LR} implies that the log likelihood ratio converges in distribution to $\caN(\rho_L, 2\rho_L)$ under $\bsH_1$, which is a consequence of Le Cam's third lemma. It also implies that the sum of the Type-I error and the Type-II error of the likelihood ratio test converges to 
\[
	\mathrm{erfc} (\sqrt{\rho_L}/2) = \frac{2}{\sqrt{\pi}} \int_{\sqrt{\rho_L}/2}^\infty e^{-x^2} \dd x.
\]
(See, e.g., Corollary 2.5 in \cite{chung2025asymptotic}.)
\end{rem}

\subsection{Threshold for the high temperature regime} \label{subsec:threshold}

Recall that in Assumption \ref{assump:second_moment}, we assumed that
\[
	\E_{\bsx, \bsx' \sim \caX} \left[ \indi(\Omega_N (\bsx)) \indi(\Omega_N (\bsx')) \exp \left( \frac{N \SNR_c \langle \bsx, \bsx' \rangle^2}{2} \right) \right]
\]
is bounded. The quantity above corresponds to the second moment of the likelihood ratio for i.i.d. prior on the event $\Omega_N$ with GOE disorder, i.e., 
\[
	\E_\caW \left[ \E_\caX \left[ \indi(\Omega_N) \prod_{i<j} \frac{p(\sqrt{N} W_{ij} - \sqrt{{\SNR_c} N} x_i x_j)}{p(\sqrt{N} W_{ij})} \prod_k \frac{p_d(\sqrt{N} W_{kk} - \sqrt{{\SNR_c} N} x_k^2)}{p_d(\sqrt{N} W_{kk})} \right]^2 \right]
\]
with $p$ and $p_d$ are the densities of the distribution $\caN(0, 1)$ and $\caN(0, 2)$, respectively. (See \eqref{eq:LR_iid} for the definition of $\ol\caL(\SNR)$.) The derivation is straightforward; see, e.g., Proposition 3.4 in \cite{Perry2018}.

The second moment has been extensively investigated in the literature to prove the contiguity, introduced by Le Cam \cite{cam1960locally}, which means that one distribution is asymptotically absolutely continuous with respect to the other and also implies that two distributions cannot be consistently distinguished. (For the precise definition of the contiguity, we refer to Definition 6.3 in \cite{van2000asymptotic}.) Notable examples are the impossibility results for the Gaussian hidden clique problem \cite{montanari2016limitation}, the clustering in the sparse stochastic block model \cite{mossel2015reconstruction,banks2016information}, the optimality of PCA for the spiked Wigner model \cite{Perry2018}, and the computationally efficient test for the same model \cite{moitra2025precise}. 

The boundedness of the second moment as in Assumption \ref{assump:second_moment} was examined in detail in \cite{Perry2018}. An important case is when the prior is sub-Gaussian for which we use Definition \ref{defn:sub_Gaussian}.

From the definition of the sub-Gaussian random variables, Definition \ref{defn:sub_Gaussian}, it can be readily checked that if $z_1, z_2, \dots, z_N$ are i.i.d. sub-Gaussian with the variance proxy $\sigma^2$, then for any $\bsv = (v_1, v_2, \dots, v_N) \in \R^N$, independent of $\bsz = (z_1, z_2, \dots, z_N)$,
\[
	\E[ e^{\langle \bsv, \bsz \rangle}] = \prod_{k=1}^N \E[e^{v_k z_k}] \leq e^{\sigma^2 \| \bsv \|^2 /2}.
\]
Then, Proposition 3.8 in \cite{Perry2018} asserts that Assumption \ref{assump:second_moment} holds if $\SNR_c < 1/\sigma^2$. In some cases, it is possible to improve the threshold by conditioning (see, e.g., \cite{banks2016information,Perry2018}), which in principle means to find a proper event $\Omega_N$ with which Assumption \ref{assump:second_moment} holds with a larger $\SNR_c$.

For the case where the disorder is Gaussian and the prior is bounded i.i.d., the optimal threshold below which Theorem \ref{thm:log_LR} for $\ol \caL (\SNR)$, known as the reconstruction threshold, was proved in \cite{AlaouiJordan2018}; see Section 2 in \cite{AlaouiJordan2018} for the precise formula for the reconstruction threshold $\lambda_c$. From Theorem \ref{thm:log_LR}, it is clear that if Assumption \ref{assump:second_moment} holds for $\SNR_c$ then $\SNR_c \leq \lambda_c$. It is not certain whether the reconstruction threshold is model-specific or universal, and we do not pursue it further in this paper.

\subsection{More general result} \label{subsec:general}

Our proof of Theorems \ref{thm:free_energy} and \ref{thm:log_LR} is based on the following proposition.
\begin{prop} \label{prop:main}
Suppose that the random variables $\{ P^{(n)}_{ij} \}_{i<j}$ ($n=1, 2, 3, 4$) satisfy the following:
\begin{itemize}
\item For any $n$, $P^{(n)}_{ij}$ are i.i.d., and for any integer $s>2$, there exists a constant $C^P_s$ such that $\E[|P^{(n)}_{ij}|^s] \leq C^P_s$.
\item For any $n$, $\E[P^{(n)}_{ij}] = 0$ and $\E[(P^{(n)}_{ij})^2] = \cn_n$ for some ($N$-independent) constants $\cn_n$.
\item $\E[P^{(1)}_{ij} P^{(2)}_{ij}] = \E[P^{(2)}_{ij} P^{(3)}_{ij}] = \E[P^{(3)}_{ij} P^{(4)}_{ij}] = \E[P^{(1)}_{ij} P^{(4)}_{ij}] = 0$ and $\E[P^{(1)}_{ij} (P^{(2)}_{ij})^2] = 0$.
\end{itemize}
Similarly, we suppose that the random variables $\{ P^{(n)}_{d,kk} \}_{1\le k \le N}$ ($n=1, 2$) satisfy the following:
\begin{itemize}
\item For any $n$, $P^{(n)}_{d, kk}$ are i.i.d., and for any integer $s>2$, there exists a constant $C^P_{d, s}$ such that $\E[|P^{(n)}_{d, kk}|^s] \leq C^P_{d, s}$.
\item For any $n$, $\E[P^{(n)}_{d,kk}] = 0$ and $\E[(P^{(n)}_{d,kk})^2] = \cn'_n$ for some ($N$-independent) constants $\cn'_n$.
\item $\E[P^{(1)}_{d,kk} P^{(2)}_{d,kk}] = 0$.
\end{itemize}
In addition, we also assume that for any $n=1, 2, 3, 4$, $m=1, 2$, $i<j$ and $k$, $P^{(n)}_{ij}$ and $P^{(m)}_{d,kk}$ are independent.
Fix $\epsilon \in (0, 1/60)$ and let
\[
	\Omega_\caX^\epsilon := \{ \max_{1 \leq k \leq N} |x_k| \leq N^{-1/2 +\epsilon} \} \cap \{ | \| \bsx \| - 1 | \leq N^{-1/5} \}.
\]
Suppose that Assumption \ref{assump:second_moment} holds with $\SNR_c$. If $\F < \SNR_c$,
\beq \begin{split} \label{eq:convergence_iid}
	&\log \E_\caX \left[ \indi(\Omega_\caX^\epsilon) \prod_{i<j} \left( 1 + P^{(1)}_{ij} \sqrt{N} x_i x_j + P^{(2)}_{ij} N x_i^2 x_j^2 + P^{(3)}_{ij} N^{3/2} x_i^3 x_j^3 + P^{(4)}_{ij} N^2 x_i^4 x_j^4 \right) \right. \\
	&\qquad \qquad \left. \times \prod_{k=1}^N \left( 1 + P^{(1)}_{d,kk} \sqrt{N} x_k^2 + P^{(2)}_{d,kk} N x_k^4 \right) \right] \Rightarrow \caN(-\rho, 2\rho)
\end{split} \eeq
as $N \to \infty$, where
\[
	\rho:= -\frac{1}{4} \left( \log (1-\F) + (\F - 2 \F') + \frac{\F^2}{2} \right) + \frac{\G}{4}.
\]
Similarly, under the same assumption
\beq \begin{split} \label{eq:convergence_normalized}
	&\log \E_\caX \left[ \indi(\Omega_\caX^\epsilon) \prod_{i<j} \left( 1 + P^{(1)}_{ij} \frac{\sqrt{N} x_i x_j}{\| \bsx \|^2} + P^{(2)}_{ij} \frac{N x_i^2 x_j^2}{\| \bsx \|^4} + P^{(3)}_{ij} \frac{N^{3/2} x_i^3 x_j^3}{\| \bsx \|^6} + P^{(4)}_{ij} \frac{N^2 x_i^4 x_j^4}{\| \bsx \|^8} \right) \right. \\
	&\qquad \qquad \left. \times \prod_{k=1}^N \left( 1 + P^{(1)}_{d,kk} \frac{\sqrt{N} x_k^2}{\| \bsx \|^2} + P^{(2)}_{d,kk} \frac{N x_k^4}{\| \bsx \|^4} \right) \right] \Rightarrow \caN(-\rho, 2\rho)
\end{split} \eeq
as $N \to \infty$.
\end{prop}

We prove Proposition \ref{prop:main} in Section \ref{sec:main_proof} by using the multigraph expansion. Note that we do not assume any special relations among the coefficients $P^{(n)}_{ij}$ in Proposition \ref{prop:main}, except some orthogonality conditions. The orthogonality condition, the third assumption in Proposition \ref{prop:main}, is satisfied if there exist symmetric, i.i.d. random variables $V_{ij}$ such that $P^{(1)}_{ij}$ and $P^{(3)}_{ij}$ are odd functions of $V_{ij}$ and $P^{(2)}_{ij}$ and $P^{(4)}_{ij}$ are even functions of $V_{ij}$.

We also remark that the parameter $G_p$ in Theorem \ref{thm:log_LR} corresponds to the fourth moment $w_4$ of the disorder in Theorem \ref{thm:free_energy}, since both stem from the same term in Proposition \ref{prop:main}, the second moment $\G$ of $P_{ij}^{(2)}$.

\section{Proof of main results} \label{sec:thm_proof}

In this section, we prove our main results. We first notice that Theorem \ref{thm:log_LR_1} and Theorem \ref{thm:free_energy_1} are simple corollaries of the general results, Theorem \ref{thm:log_LR} and Theorem \ref{thm:free_energy}.

\begin{proof}[Proof of Theorem \ref{thm:log_LR_1} and Theorem \ref{thm:free_energy_1}]
Theorem \ref{thm:log_LR_1} and Theorem \ref{thm:free_energy_1} are special cases of Theorem \ref{thm:log_LR} and Theorem \ref{thm:free_energy}, respectively, since by definition the strictly sub-Gaussian prior is the sub-Gaussian prior with variance proxy $1$.
\end{proof}

In the rest of the section, we prove Theorem \ref{thm:log_LR} and Theorem \ref{thm:free_energy} by applying Proposition \ref{prop:main}. We first introduce the following shorthand notations, which will be frequently used in the rest of the paper.
\begin{nrem}
	We use the standard big-$O$ and little-$o$ notation: $a_N = O(b_N)$ implies that there exists $N_0$ such that $|a_N| \leq C |b_N|$ for some constant $C>0$ independent of $N$ for all $N \geq N_0$; $a_N = o(b_N)$ implies that for any positive constant $\varepsilon$ there exists $N_0$ such that $|a_N| \leq \varepsilon |b_N|$ for all $N \geq N_0$. 
	
	For probabilistic bounds, we use the notation $X=o_{\caW}(Y)$ or $X=o_{\caX}(Y)$, which implies that for any (small) $\varepsilon >0$, $\p_\caW(|X|>\varepsilon|Y|)\to 0$ or $\p_\caX(|X|>\varepsilon|Y|)\to 0$ as $N\to \infty$, respectively.
	
  For an event $\Omega$, we say that $\Omega$ holds with overwhelming probability if for any (large) $D > 0$ there exists $N_0 \equiv N_0 (D)$ such that $\p(\Omega^c) < N^{-D}$ whenever $N > N_0$. Again, we say $\Omega$ holds with overwhelming probability with respect to $\caW$ (or $\caX$) if $\p_\caW(\Omega^c) < N^{-D}$ (or $\p_\caX(\Omega^c) < N^{-D}$).
	
	For an event $\Omega$, we say that $\Omega$ holds with high probability if $\p(\Omega^c) = o(1)$. Again, we say $\Omega$ holds with high probability with respect to $\caW$ (or $\caX$) if $\p_\caW(\Omega^c) = o(1)$ (or $\p_\caX(\Omega^c) = o(1)$).
	
	For a sequence of random variables, the notation $\Rightarrow$ denotes the convergence in distribution as $N\rightarrow\infty$.
	
	For integers $a$ and $b$, the notation $[[a, b]]$ denotes $\{ n \in \mathbb{Z}: a \leq n \leq b \}$.
\end{nrem}

\subsection{Proof of Theorem \ref{thm:log_LR}} \label{subsec:log_LR}

We begin by introducing a lemma that we will use frequently to neglect rare events in $\caX$.
\begin{lem} \label{lem:truncate}
	Suppose that $R_N$ is a non-negative random variable depending on $W$ and $\bsx$. If $\Omega_N \equiv \Omega_N(\bsx)$ is an event in $\caX$ with $\p_{\caX}(\Omega_N^c) = o(a_N^{-1})$ for some deterministic $a_N$ and $\E_\caW[R_N] =O(a_N)$ almost surely with respect to $\caX$, then
\beq \label{eq:truncate}
\E_\caX [R_N] - \E_\caX [ \indi(\Omega_N) R_N ] = o_\caW(1).
\eeq
\end{lem}

\begin{proof}[Proof of Lemma \ref{lem:truncate}]
Taking the partial expectation $\E_\caW$ to $\E_\caX [R_N] - \E_\caX [ \indi(\Omega_N) R_N ] = \E_\caX [ \indi(\Omega_N^c) R_N ]$,
\[
	\E_\caW \E_\caX [ \indi(\Omega_N^c) R_N ] = \E_\caX [ \indi(\Omega_N^c) \E_\caW[R_N] ] \leq C a_N \E_\caX [\indi(\Omega_N^c)] = C a_N \p_\caX(\Omega_N^c)
\]
for some constant $C$. Since the right side is $o(1)$, we get the desired lemma by applying Markov's inequality in $\caW$.
\end{proof}

We now prove Theorem \ref{thm:log_LR}.

\begin{proof}[Proof of Theorem \ref{thm:log_LR}]
Recall the definition of the likelihood ratio $\ol \caL(\SNR)$ in Definition \ref{defn:LR_iid}. Note that the entries of $W$ are independent. Since
\[ \begin{split}
	\E_\caW \left[ \frac{p(\sqrt{N} W_{ij} - \sqrt{\SNR N} x_i x_j)}{p(\sqrt{N} W_{ij})} \right] &= \int_{-\infty}^{\infty} \frac{p(w - \sqrt{\SNR N} x_i x_j)}{p(w)} \, p(w) \dd w \\
	&= \int_{-\infty}^{\infty} p(w - \sqrt{\SNR N} x_i x_j) \dd w = 1,
\end{split} \]
and, similarly, $\E_\caW \left[ \frac{p_d(\sqrt{N} W_{kk} - \sqrt{\SNR N} x_k^2)}{p_d(\sqrt{N} W_{kk})} \right] = 1$, we find that $\E_\caW [\ol \caL(\SNR)] = 1$, and hence we can apply Lemma \ref{lem:truncate} in the proof of Theorem \ref{thm:log_LR} to neglect any events in $\caX$ with probability $o(1)$.

Let
\beq \label{eq:Omega_definition}
	\Omega_\caX := \{ \max_{1 \leq k \leq N} |x_k| \leq N^{-1/2 +\eta} \}, \quad \Omega_\caW := \{ \max_{i<j} |W_{ij}| \leq N^{-1/2 +\eta} \}
\eeq
for some sufficiently small $\eta > 0$, which will be explained later. From (high-order) Markov's inequality, we find that $\Omega_\caX$ holds with overwhelming probability with respect to $\caX$ and $\Omega_\caW$ also holds with overwhelming probability with respect to $\caW$. Expanding the product in the right side of \eqref{eq:LR_iid} on $\Omega_\caW$, for the off-diagonal terms, on $\Omega_\caX$, we get
\[ \begin{split}
	&\indi(\Omega_\caX) \prod_{i<j} \frac{p(\sqrt{N} W_{ij} - \sqrt{\SNR N} x_i x_j)}{p(\sqrt{N} W_{ij})}  \\
	&= \indi(\Omega_\caX) \prod_{i<j} \left( 1 + p^{(1)}_{ij} \sqrt{N} x_i x_j + p^{(2)}_{ij} N x_i^2 x_j^2 + p^{(3)}_{ij} N^{3/2} x_i^3 x_j^3 + p^{(4)}_{ij} N^2 x_i^4 x_j^4  \right) + o_\caW(1),  
\end{split} \]
where we used the shorthand notation
\[
	p^{(n)}_{ij} = \frac{(-1)^n \SNR^{n/2} p^{(n)}(\sqrt{N} W_{ij})}{n! p(\sqrt{N} W_{ij})} \qquad (n=1, 2, 3, 4).
\]
Similarly,
\beq \begin{split}
	\indi(\Omega_\caX) \prod_{k=1}^N \frac{p_d(\sqrt{N} W_{kk} - \sqrt{\SNR N} x_k^2)}{p_d(\sqrt{N} W_{kk})} = \indi(\Omega_\caX) \prod_{k=1}^N \left( 1 + p^{(1)}_{d,kk} \sqrt{N} x_k^2 + p^{(2)}_{d,kk} N x_k^4 \right) + o_\caW(1),
\end{split} \eeq
with
\[
	p^{(n)}_{d, kk} = \frac{(-1)^n \SNR^{n/2} p_d^{(n)}(\sqrt{N} W_{kk})}{n! p_d(\sqrt{N} W_{kk})} \qquad (n=1, 2).
\]
We need to choose $\eta$ so small that, assuming $\Omega_\caX$ and $\Omega_\caW$,
\[
	\sum_{i<j}|p^{(5)}(x)/p(x)| N^{5/2} x_i^5 x_j^5 \leq N^{-1/10}, \quad \sum_k|p_d^{(3)}(y)/p_d(y)| N^{3/2} x_k^6 \leq N^{-1/10}
\]
for any $x \in [\sqrt{N} W_{ij} - \sqrt{\SNR N} x_i x_j, \sqrt{N} W_{ij}]$ and any $y \in [\sqrt{N} W_{kk} - \sqrt{\SNR N} x_k^2, \sqrt{N} W_{kk}]$, uniformly in $i, j$, and $k$. (Note that this is possible due to the assumption that the functions $p^{(s)}/p$ and $p_d^{(s)}/p_d$ are polynomially bounded in Assumption \ref{assump:decay}.)

Since $\Omega_\caX^c$ and $(\Omega_\caX^\epsilon)^c$ are the events in $\caX$ with probability $o(1)$, we now find that
\[ \begin{split}
	\ol \caL(\SNR)
	&= \E_\caX \left[ \indi(\Omega_\caX) \indi(\Omega_\caX^\epsilon) \prod_{i<j} \left( 1 + p^{(1)}_{ij} \sqrt{N} x_i x_j + p^{(2)}_{ij} N x_i^2 x_j^2 + p^{(3)}_{ij} N^{3/2} x_i^3 x_j^3 + p^{(4)}_{ij} N^2 x_i^4 x_j^4 \right) \right. \\
	&\qquad \qquad \left. \times \prod_{k=1}^N \left( 1 + p^{(1)}_{d, kk} \sqrt{N} x_k^2 + p^{(2)}_{d, kk} N x_k^4 \right) \right] + o_\caW(1).
\end{split} \]
Note that $\E[p^{(n)}_{ij}] = 0$ for $n=1, 2, 3, 4$ and $\E[p^{(n)}_{d, kk}] = 0$ for $n=1, 2$. Applying Lemma \ref{lem:truncate} again, we get
\[ \begin{split}
	\ol \caL(\SNR)
	&= \E_\caX \left[\indi(\Omega_\caX^\epsilon) \prod_{i<j} \left( 1 + p^{(1)}_{ij} \sqrt{N} x_i x_j + p^{(2)}_{ij} N x_i^2 x_j^2 + p^{(3)}_{ij} N^{3/2} x_i^3 x_j^3 + p^{(4)}_{ij} N^2 x_i^4 x_j^4 \right) \right. \\
	&\qquad \qquad \left. \times \prod_{k=1}^N \left( 1 + p^{(1)}_{d, kk} \sqrt{N} x_k^2 + p^{(2)}_{d, kk} N x_k^4 \right) \right] + o_\caW(1).
\end{split} \]
Since $p^{(n)}_{ij}$ and $p^{(n)}_{d, kk}$ are odd functions if $n$ is odd, and even functions if $n$ even, we find that we can apply Proposition \ref{prop:main} if $\F = \SNR F_p < \SNR_c$. The mean and the variance of the limiting Gaussian can be readily computed from that $\G = \SNR^2 G_p/4$ and $\F' =\SNR F_d$. This proves that Theorem \ref{thm:log_LR} holds for $\ol \caL(\SNR)$ if Assumption \ref{assump:second_moment} holds. Since Assumption \ref{assump:second_moment} holds if the prior is sub-Gaussian with the variance proxy $\sigma^2$ and $\SNR_c < 1/\sigma^2$, as discussed in Section \ref{subsec:threshold}, this proves the first part of Theorem \ref{thm:log_LR}.

To prove the result for $\caL(\SNR)$, we apply the same argument as above, except that we change the definition of $\Omega_\caX$ in \eqref{eq:Omega_definition} to
\[
	\Omega_\caX' := \{ \max_{1 \leq k \leq N} |x_k| \leq N^{-1/2 +\eta} \} \cap \{ | \| \bsx \| - 1 | \leq N^{-1/5} \},
\]
which allows us to expand for the off-diagonal terms
\beq \begin{split}
	&\indi(\Omega_\caX') \prod_{i<j} \frac{p(\sqrt{N} W_{ij} - \sqrt{\SNR N} x_i x_j / \| \bsx \|^2)}{p(\sqrt{N} W_{ij})}  \\
	&= \indi(\Omega_\caX') \prod_{i<j} \left( 1 + p^{(1)}_{ij} \frac{\sqrt{N} x_i x_j}{\| \bsx \|^2} + p^{(2)}_{ij} \frac{N x_i^2 x_j^2}{\| \bsx \|^4} + p^{(3)}_{ij} \frac{N^{3/2} x_i^3 x_j^3}{\| \bsx \|^6} + p^{(4)}_{ij} \frac{N^2 x_i^4 x_j^4}{\| \bsx \|^8} \right) + o_\caW(1),  
\end{split} \eeq
and similarly for the diagonal terms. Since $\Omega_\caX'$ holds with overwhelming probability, this proves that Theorem \ref{thm:log_LR} holds for $\caL(\SNR)$ and completes the proof of Theorem \ref{thm:log_LR}.
\end{proof}

\subsection{Proof of Theorem \ref{thm:free_energy}} \label{subsec:free_energy}

We next prove Theorem \ref{thm:free_energy}. Heuristically, the partition function $Z$ is of order $e^{F(\beta)N}$, which means that we cannot simply assume a high probability event in $\caX$, or even an overwhelming probability event. This property makes the justification of the Taylor expansion of the partition function significantly more complicated than that of the log likelihood we saw in the proof of Theorem \ref{thm:log_LR}. In the first part of the proof of Theorem \ref{thm:free_energy}, we mainly focus on how the strict sub-Gaussianity of the disorder can be applied to resolve this issue.

\begin{proof}[Proof of Theorem \ref{thm:free_energy} - Case I]
Suppose that $\sqrt{N} W_{ij}$ is strictly sub-Gaussian. Define
\beq \begin{split} \label{eq:wt_Z}
	\wt Z &:= e^{-N\beta^2 /4} \exp \left( \beta N \sum_{i<j} W_{ij} \frac{x_i x_j}{\| \bsx \|^2} \right) \\
	&= \exp \left( \beta N \sum_{i<j} W_{ij} \frac{x_i x_j}{\| \bsx \|^2}- \frac{\beta^2 N}{2} \sum_{i<j} \frac{x_i^2 x_j^2}{\| \bsx \|^4} - \frac{\beta^2 N}{4} \sum_{i=1}^N \frac{x_i^4}{\| \bsx \|^4} \right).
\end{split} \eeq
From the strictly sub-Gaussianity,
\[
	\E_\caW \left[ \exp \left( \beta N W_{ij} \frac{x_i x_j}{\| \bsx \|^2} \right) \right] \leq \exp \left( \frac{\beta^2 N}{2} \frac{x_i^2 x_j^2}{\| \bsx \|^4} \right),
\]
and thus from the independence of $\sqrt{N} W_{ij}$'s, we have
\[
	\E_\caW [ \wt Z] \leq \exp \left( - \frac{\beta^2 N}{4} \sum_{i=1}^N \frac{x_i^4}{\| \bsx \|^4} \right) \leq 1.
\]
Let $m_8 := N^4 \E [x_1^8]$ and define
\beq \begin{split} \label{eq:Omega''}
	\Omega'' &:= \{ \max_{1 \leq k \leq N} |x_k| \leq N^{-29/60} \} \cap \{ | \| \bsx \| - 1 | \leq N^{-1/5} \} \cap \big\{ \big| N \sum_i x_i^4 - m_4 \big| \leq N^{-1/5} \big\} \\
	&\qquad \qquad \cap \big\{ \big| N^3 \sum_i x_i^8 - m_8 \big| \leq N^{-1/5} \big\}.
\end{split} \eeq
Since $\Omega''$ and $\Omega_\caX^\epsilon$ hold with overwhelming probability, we find from Lemma \ref{lem:truncate} that 
\[
	e^{-N\beta^2 /4} Z_N = \E_\caX[\wt Z] = \E_\caX[\indi(\Omega'') \indi(\Omega_\caX^\epsilon) \wt Z] + o_\caW(1).
\]

We next expand the sum in the exponent in \eqref{eq:wt_Z}. Let $\Omega_\caW := \{ \max_{i<j} |W_{ij}| \leq N^{-29/60} \}$. Following the argument in the proof of Theorem \ref{thm:log_LR},
\beq \begin{split} \label{eq:P_free_energy_start}
	&\indi(\Omega'') \indi(\Omega_\caX^\epsilon) \exp \left( \beta N \sum_{i<j} W_{ij} \frac{x_i x_j}{\| \bsx \|^2} - \frac{\beta^2 N}{2} \sum_{i<j} \frac{x_i^2 x_j^2}{\| \bsx \|^4} - \frac{\beta^2 N}{4} \sum_{i=1}^N \frac{x_i^4}{\| \bsx \|^4} \right) \\
	&= \indi(\Omega'') \indi(\Omega_\caX^\epsilon) \indi(\Omega_\caW) \exp \left( \beta N \sum_{i<j} W_{ij} \frac{x_i x_j}{\| \bsx \|^2} - \frac{\beta^2 N}{2} \sum_{i<j} \frac{x_i^2 x_j^2}{\| \bsx \|^4} - \frac{\beta^2 N}{4} \sum_{i=1}^N \frac{x_i^4}{\| \bsx \|^4} \right) + o_\caW(1) \\
	&= \indi(\Omega'') \indi(\Omega_\caX^\epsilon) \indi(\Omega_\caW) e^{-\beta^2 m_4 /4} \prod_{i<j} \left( 1 + P^{(1)}_{ij} \frac{\sqrt{N} x_i x_j}{\| \bsx \|^2} + P^{(2)}_{ij} \frac{N x_i^2 x_j^2}{\| \bsx \|^4} + P^{(3)}_{ij} \frac{N^{3/2} x_i^3 x_j^3}{\| \bsx \|^6} + \wt P^{(4)}_{ij} \frac{N^2 x_i^4 x_j^4}{\| \bsx \|^8} \right) \\
	&\qquad + o_\caW(1)
\end{split} \eeq
with
\beq \begin{split} \label{eq:P_free_energy}
	P^{(1)}_{ij} &= \beta \sqrt{N} W_{ij}, \quad P^{(2)}_{ij} = \frac{\beta^2}{2} \left( (\sqrt{N} W_{ij})^2 -1 \right), \quad P^{(3)}_{ij} = \frac{\beta^3}{6} \left( (\sqrt{N} W_{ij})^3 -3 (\sqrt{N} W_{ij}) \right), \\
	\wt P^{(4)}_{ij} &= \frac{\beta^4}{24} \left( (\sqrt{N} W_{ij})^4 -6 (\sqrt{N} W_{ij})^2 + 3 \right).
\end{split} \eeq
Note that $\E[P^{(1)}_{ij}] = \E[P^{(2)}_{ij}] = \E[P^{(3)}_{ij}] = 0$ and $\E[\wt P^{(4)}_{ij}] = \beta^4(w_4 -3)/24$. To make each term in the product centered, we use the relation
\[ \begin{split}
	& \indi(\Omega'') \indi(\Omega_\caX^\epsilon) \indi(\Omega_\caW) \prod_{i<j} \left( 1 + P^{(1)}_{ij} \frac{\sqrt{N} x_i x_j}{\| \bsx \|^2} + P^{(2)}_{ij} \frac{N x_i^2 x_j^2}{\| \bsx \|^4} + P^{(3)}_{ij} \frac{N^{3/2} x_i^3 x_j^3}{\| \bsx \|^6} + \wt P^{(4)}_{ij} \frac{N^2 x_i^4 x_j^4}{\| \bsx \|^8} \right)  \\
	&= \indi(\Omega'') \indi(\Omega_\caX^\epsilon) \indi(\Omega_\caW) \prod_{i<j} \left( 1 + P^{(1)}_{ij} \frac{\sqrt{N} x_i x_j}{\| \bsx \|^2} + P^{(2)}_{ij} \frac{N x_i^2 x_j^2}{\| \bsx \|^4} + P^{(3)}_{ij} \frac{N^{3/2} x_i^3 x_j^3}{\| \bsx \|^6} + P^{(4)}_{ij} \frac{N^2 x_i^4 x_j^4}{\| \bsx \|^8} \right) \\
	&\qquad \qquad \qquad \times \left( 1 + \frac{\beta^4 (w_4-3)}{24} \frac{N^2 x_i^4 x_j^4}{\| \bsx \|^8} \right) + o_\caW(1),
\end{split} \]
which replaces $\wt P^{(4)}_{ij}$ in \eqref{eq:P_free_energy} by
\[
P^{(4)}_{ij} = \frac{\beta^4}{24}\left((\sqrt{N} W_{ij})^4 -6 (\sqrt{N} W_{ij})^2 + 6 -w_4\right).
\]
We can easily check the orthogonality conditions of $P_{ij}^{(n)}$.
Note that the additional factor can be approximated by
\[ \begin{split}
	&\indi(\Omega'') \indi(\Omega_\caX^\epsilon) \prod_{i<j} \left( 1 + \frac{\beta^4 (w_4-3)}{24} \frac{N^2 x_i^4 x_j^4}{\| \bsx \|^8} \right) = (1+O(N^{-1})) \indi(\Omega'') \prod_{i<j} \exp \left( \frac{\beta^4 (w_4-3)}{24} \frac{N^2 x_i^4 x_j^4}{\| \bsx \|^8} \right)  \\
	&= (1+O(N^{-1})) \indi(\Omega'') \indi(\Omega_\caX^\epsilon) \exp \left( \frac{\beta^4 (w_4-3)}{24} N^2 \sum_{i<j} \frac{x_i^4 x_j^4}{\| \bsx \|^8} \right) \\
	&= (1+O(N^{-1/5})) \indi(\Omega'') \indi(\Omega_\caX^\epsilon) \exp \left( \frac{\beta^4 (w_4-3) m_4^2}{48} \right).
\end{split} \]

So far, we have seen that
\[ \begin{split}
	&e^{-N\beta^2 /4} Z_N = \E_\caX[\indi(\Omega'') \indi(\Omega_\caX^\epsilon) \wt Z] + o_\caW(1) \\
	&=\E_\caX \Bigg[ \indi(\Omega'') \indi(\Omega_\caX^\epsilon) \indi(\Omega_\caW) e^{-\beta^2 m_4 /4} (1+O(N^{-1/5})) \exp \left( \frac{\beta^4 (w_4-3) m_4^2}{48} \right) \\
	&\qquad \qquad \times \prod_{i<j} \left( 1 + P^{(1)}_{ij} \frac{\sqrt{N} x_i x_j}{\| \bsx \|^2} + P^{(2)}_{ij} \frac{N x_i^2 x_j^2}{\| \bsx \|^4} + P^{(3)}_{ij} \frac{N^{3/2} x_i^3 x_j^3}{\| \bsx \|^6} + P^{(4)}_{ij} \frac{N^2 x_i^4 x_j^4}{\| \bsx \|^8} \right) \Bigg] + o_\caW(1) \\
	&=\indi(\Omega_\caW) e^{-\beta^2 m_4 /4} (1+O(N^{-1/5})) \exp \left( \frac{\beta^4 (w_4-3) m_4^2}{48} \right)  \\
	&\quad \times \E_\caX \Bigg[ \indi(\Omega'') \indi(\Omega_\caX^\epsilon) \prod_{i<j} \left( 1 + P^{(1)}_{ij} \frac{\sqrt{N} x_i x_j}{\| \bsx \|^2} + P^{(2)}_{ij} \frac{N x_i^2 x_j^2}{\| \bsx \|^4} + P^{(3)}_{ij} \frac{N^{3/2} x_i^3 x_j^3}{\| \bsx \|^6} + P^{(4)}_{ij} \frac{N^2 x_i^4 x_j^4}{\| \bsx \|^8} \right) \Bigg] + o_\caW(1),
\end{split} \]
which shows that 
\[ \begin{split}
	&(1+O(N^{-1/5})) \exp \left( -\frac{N \beta^2}{4} + \frac{\beta^2 m_4}{4} - \frac{\beta^4 (w_4-3) m_4^2}{48} \right) Z_N  \\
	& =\E_\caX \left[ \indi(\Omega'') \indi(\Omega_\caX^\epsilon) \prod_{i<j} \left( 1 + P^{(1)}_{ij} \frac{\sqrt{N} x_i x_j}{\| \bsx \|^2} + P^{(2)}_{ij} \frac{N x_i^2 x_j^2}{\| \bsx \|^4} + P^{(3)}_{ij} \frac{N^{3/2} x_i^3 x_j^3}{\| \bsx \|^6} + P^{(4)}_{ij} \frac{N^2 x_i^4 x_j^4}{\| \bsx \|^8} \right) \right] + o_\caW(1).
\end{split} \]
Now, as in the proof of Theorem \ref{thm:log_LR}, we can apply Proposition \ref{prop:main} to find
\[ \begin{split}
\left( -\frac{N \beta^2}{4} + \frac{\beta^2 m_4}{4} - \frac{\beta^4 (w_4-3) m_4^2}{48} \right) + \log Z \Rightarrow \caN(-\rho, 2\rho),
\end{split} \]
which proves that $N(F_N(\beta) - F(\beta))$ converges in distribution to a Gaussian if Assumption \ref{assump:second_moment} holds. 

It remains to find the mean and the variance of the limiting distribution of $N(F_N(\beta) - F(\beta))$. Since
\[
	\F = \beta^2, \quad \G = \beta^4 (w_4 -1)/4, \quad \F' = 0,
\]
it can be readily checked that $\rho$ in Proposition \ref{prop:main} is given by
\[
	\rho = -\frac{1}{4} \left( \log (1- \beta^2) + \beta^2 - \frac{\beta^4 (w_4 -3)}{4} \right),
\]
and the variance $V_F = 2\rho$. Finally, the mean $m_F$ is given by
\[
	m_F = -\rho + \left( \frac{w_4-3}{48} \beta^4 m_4^2 - \frac{\beta^2}{4} m_4 \right) = \frac{1}{4} \log (1- \beta^2) + \frac{\beta^2 (1-m_4)}{4} + \frac{\beta^4 (w_4-3)(m_4^2 -3)}{48}.
\]
This proves Theorem \ref{thm:free_energy} when $\sqrt{N} W_{ij}$ is strictly sub-Gaussian under Assumption \ref{assump:second_moment}. Since Assumption \ref{assump:second_moment} holds if the prior is sub-Gaussian with the variance proxy $\sigma^2$ and $\SNR_c < 1/\sigma^2$, as discussed in Section \ref{subsec:threshold}, this completes the proof of Theorem \ref{thm:free_energy}.
\end{proof}

In the second part of the proof of Theorem \ref{thm:free_energy}, we use a good property of the prior to justify the Taylor expansion.

\begin{proof}[Proof of Theorem \ref{thm:free_energy} - Case II]
Define
\[
	\Omega^{(0)} = \{ \| \bsx \| \geq N^{-1/40} \}.
\]
Let $N_i := \{ i: x_i^2 \geq N^{-51/50} \}$. If $|N_i| > N^{49/50}$, then
\[
	\sum_i x_i^2 \geq \sum_{i \in N_i} x_i^2 \geq N^{49/50} N^{-51/50} > N^{-1/20},
\]
and $\Omega^{(0)}$ holds. Hence,
\beq \label{eq:Omega^0_bound_1}
	\p( (\Omega^{(0)})^c ) \leq \sum_{r=0}^{\lfloor N^{49/50} \rfloor} \p (|N_i| = r).
\eeq
Since the distribution of $\sqrt{N} x_i$ is absolutely continuous in a neighborhood of $0$, there exists a constant $C_0$ such that
\[
	\p( x_i^2 < N^{-51/50} ) = \p( \sqrt{N}|x_i| < N^{-1/100} ) \leq C_0 N^{-1/100}.
\]
Thus,
\[
	\p (|N_i| = r) = \binom{N}{r} \p( x_i^2 < N^{-51/50} )^{N-r} \p( x_i^2 \geq N^{-51/50} )^r \leq \frac{N^r}{r!} C_0^{N-r} N^{-(N-r)/100},
\]
and from \eqref{eq:Omega^0_bound_1} we get
\[
	\p( (\Omega^{(0)})^c ) \leq \sum_{r=0}^{\lfloor N^{49/50} \rfloor} \frac{N^r}{r!} C_0^{N-r} N^{-(N-r)/100} < N^{(101/100) N^{49/50}} e^{1/C_0} C_0^N N^{-N/100}.
\]
This in particular shows that $\p( (\Omega^{(0)})^c ) = o(C^{-N})$ for any constant $C>0$.

Recall the definition of $\wt Z$ in \eqref{eq:wt_Z}. It is a well-known fact from random matrix theory that for any $\delta > 0$, $\| W \| \leq 2+\delta$ with overwhelming probability. (See, e.g., \cite{erdHos2012rigidity}.) Thus, letting 
\[
	\wt \Omega_\caW := \{ \| W \| < 2+\delta \} \cap \{ \max_{i<j} |W_{ij}| \leq N^{-29/60} \},
\]
we have 
\[
	\indi(\wt \Omega_\caW) \left| \beta N \sum_{i<j} W_{ij} \frac{x_i x_j}{\| \bsx \|^2} \right| \leq \indi(\wt \Omega_\caW) \frac{\beta(2+\delta) N}{2}
\]
where the factor $1/2$ is due to that we only consider $i<j$ in the sum in the partition function. In particular, $\E_\caW[\indi(\wt \Omega_\caW) \wt Z] \leq e^{C_1 N}$ for some constant $C_1 > 0$. We now find from Lemma \ref{lem:truncate} that
\[
	e^{-N\beta^2 /4} Z_N= \E_\caX[\indi(\wt \Omega_\caW) \wt Z] + o_\caW(1) = \E_\caX[\indi(\Omega^{(0)}) \indi(\wt \Omega_\caW) \wt Z] + o_\caW(1).
\]

Expanding the sum in the exponent in the definition of $\wt Z$ in \eqref{eq:wt_Z} on $\Omega^{(0)}$ and $\wt \Omega_\caW$,
\[ \begin{split}
	&\E_\caX [\indi(\Omega^{(0)}) \indi(\wt \Omega_\caW) \wt Z] \\
	&= \E_\caX \left[\indi(\Omega^{(0)}) \indi(\wt \Omega_\caW) \prod_{i<j} \left( 1 +P^{(1)}_{ij} \frac{\sqrt{N} x_i x_j}{\| \bsx \|^2} + P^{(2)}_{ij} \frac{N x_i^2 x_j^2}{\| \bsx \|^4} + P^{(3)}_{ij} \frac{N^{3/2} x_i^3 x_j^3}{\| \bsx \|^6} + \wt P^{(4)}_{ij} \frac{N^2 x_i^4 x_j^4}{\| \bsx \|^8} \right) \right. \\
	&\qquad \qquad \times \left. \exp \left( -\frac{\beta^2 N}{4} \sum_{i=1}^N \frac{x_i^4}{\| \bsx \|^4} \right) \right] + o_\caW(1) \\
	&= \E_\caX \left[\indi(\Omega^{(0)}) \prod_{i<j} \left( 1 + P^{(1)}_{ij} \frac{\sqrt{N} x_i x_j}{\| \bsx \|^2} + P^{(2)}_{ij} \frac{N x_i^2 x_j^2}{\| \bsx \|^4} + P^{(3)}_{ij} \frac{N^{3/2} x_i^3 x_j^3}{\| \bsx \|^6} + \wt P^{(4)}_{ij} \frac{N^2 x_i^4 x_j^4}{\| \bsx \|^8} \right) \right. \\
	&\qquad \qquad \times \left. \exp \left( -\frac{\beta^2 N}{4} \sum_{i=1}^N \frac{x_i^4}{\| \bsx \|^4} \right) \right] + o_\caW(1) 	
\end{split} \]
where $P^{(n)}_{ij}$ is as defined in \eqref{eq:P_free_energy}. Taking partial expectation $\E_\caW$, since $|\sqrt{N} x_i| \leq K$ for some constant $K$,
\[ \begin{split}
	&\E_\caW \left[ \indi(\Omega^{(0)}) \prod_{i<j} \left( 1 + P^{(1)}_{ij} \frac{\sqrt{N} x_i x_j}{\| \bsx \|^2} + P^{(2)}_{ij} \frac{N x_i^2 x_j^2}{\| \bsx \|^4} + P^{(3)}_{ij} \frac{N^{3/2} x_i^3 x_j^3}{\| \bsx \|^6} + \wt P^{(4)}_{ij} \frac{N^2 x_i^4 x_j^4}{\| \bsx \|^8} \right) \right] \\
	&\leq \prod_{i<j} \left( 1 + \frac{\beta^4(w_4 -3) K^8}{24} N^{-9/5} \right) \leq C^{N^{1/5}}
\end{split} \]
for some constant $C$. Define
\[
	\Omega^{(1)} := \{ | \| \bsx \|^2 - 1 | \leq N^{-1/5} \} \cap \big\{ \big| N \sum_i x_i^4 - m_4 \big| \leq N^{-1/5} \big\} \cap \big\{ N \sum_i x_i^8 \leq N^{-1} \big\}.
\]
Since $|\sqrt{N} x_i|$ is bounded, by Hoeffding's inequality, $\p( (\Omega^{(1)})^c ) < c_0^{-N^{3/5}}$ for some constant $c_0>1$. Thus, from Lemma \ref{lem:truncate},
\[ \begin{split}
	&\E_\caX [\indi(\Omega^{(0)}) \indi(\wt \Omega_\caW) \wt Z] \\
	&= \E_\caX \left[\indi(\Omega^{(1)}) \indi(\Omega^{(0)}) \prod_{i<j} \left(1 + P^{(1)}_{ij} \frac{\sqrt{N} x_i x_j}{\| \bsx \|^2} + P^{(2)}_{ij} \frac{N x_i^2 x_j^2}{\| \bsx \|^4} + P^{(3)}_{ij} \frac{N^{3/2} x_i^3 x_j^3}{\| \bsx \|^6} + \wt P^{(4)}_{ij} \frac{N^2 x_i^4 x_j^4}{\| \bsx \|^8} \right) \right. \\
	&\qquad \qquad \times \left. \exp \left( -\frac{\beta^2 N}{4} \sum_{i=1}^N \frac{x_i^4}{\| \bsx \|^4} \right) \right] + o_\caW(1) \\
	&= \E_\caX \left[\indi(\Omega^{(1)}) \prod_{i<j} \left( 1 + P^{(1)}_{ij} \frac{\sqrt{N} x_i x_j}{\| \bsx \|^2} + P^{(2)}_{ij} \frac{N x_i^2 x_j^2}{\| \bsx \|^4} + P^{(3)}_{ij} \frac{N^{3/2} x_i^3 x_j^3}{\| \bsx \|^6} + \wt P^{(4)}_{ij} \frac{N^2 x_i^4 x_j^4}{\| \bsx \|^8} \right) \right. \\
	&\qquad \qquad \times \left. \exp \left( -\frac{\beta^2 N}{4} \sum_{i=1}^N \frac{x_i^4}{\| \bsx \|^4} \right) \right] + o_\caW(1),
\end{split} \]
where in the last line we used that $\Omega^{(1)} \subset \Omega^{(0)}$.

We now have a result analogous to \eqref{eq:P_free_energy_start}. Following the proof for Case I, we can complete the proof for Case II. We omit the detail.
\end{proof}

\section{Multigraph expansion - Idea of proof} \label{sec:multigraph}

In this section, we introduce the multigraph expansion and the main idea of our proof of Proposition \ref{prop:main}.

\subsection{Graph expansion}

As a toy model, suppose that we consider an object of the form
\beq \label{eq:expansion_example}
	\E_\caX \left[ \prod_{i<j} \left( 1 + P^{(1)}_{ij} \sqrt{N} x_i x_j + P^{(2)}_{ij} N x_i^2 x_j^2  \right) \right],
\eeq
for certain random variables $P^{(1)}_{ij}$'s and $P^{(2)}_{ij}$'s, where $\caX$ is the Rademacher prior. Note that we can first simplify the product by using the property that $x_i^2 = x_j^2 = N^{-1}$. Further, all monomials appearing in the product $\prod_{i<j} ( 1 + P^{(1)}_{ij} \sqrt{N} x_i x_j + P^{(2)}_{ij} N x_i^2 x_j^2)$ vanishes after taking the partial expectation $\E_\caX$ except the ones whose exponents are all even. This suggests to first decompose the product into two parts, one containing the cycles with nodes corresponding to $x_1, \dots, x_N$ and the edges corresponding to $P^{(1)}_{ij}$'s, and the other an object consisting of $P^{(2)}_{ij}$'s, independent of the prior. The analysis in this case is essentially the same as in the cluster expansion. 

When the prior is not Rademacher, however, the expansion is different from the conventional cluster expansion. For example, in the expansion of \eqref{eq:expansion_example}, we can find a term 
\[
	\left( N P^{(2)}_{12} x_1^2 x_2^2 \right) \left( N P^{(2)}_{23} x_2^2 x_3^2 \right) = N^2 P^{(2)}_{12} P^{(2)}_{23} x_1^2 x_2^4 x_3^2,
\]
which is considered as a monomial of the variables $x_1, x_2, x_3$ with the coefficient $N^2 P^{(2)}_{12} P^{(2)}_{23}$. Then, even if $x_i$'s are independent and $\E[x_i^2] = 1/N$, we still have
\[
	\E_\caX\left[ N^2 P^{(2)}_{12} P^{(2)}_{23} x_1^2 x_2^4 x_3^2 \right] = P^{(2)}_{12} P^{(2)}_{23} \E_\caX[x_2^4],
\]
which means that we need to keep track of the exponents of $x_i$'s. For systematic analysis, we use a one-to-one correspondence between the monomials in the expansion of \eqref{eq:expansion_example} and graphs with $N$ nodes, where the exponent of $x_i$ in a monomial in the expansion is equal to the degree of the node $i$ in a graph. In the example above, since the exponents of the variables $x_1, x_2, x_3$ are $2, 4, 2$, respectively, we are led to consider a graph with the nodes $1, 2, 3$ with the degrees $2, 4, 2$, respectively. (See the left graph in Figure \ref{fig:butterfly_graph}.) It is inevitable that some of the nodes in the graph must be joined by double edges since the term containing $x_i^2 x_j^2$ increases the exponents by $2$ for both $x_i$ and $x_j$, or equivalently the degrees of the nodes $i$ and $j$. Since the term $x_i^2 x_j^2$ is paired with the random variable $P_{ij}^{(2)}$ in the expansion, we can set a rule that $P_{ij}^{(2)}$ is attached to the double edge between the nodes $i$ and $j$. Similarly, $P_{ij}^{(1)}$ is attached to a single edge between the nodes $i$ and $j$. In a typical graph, we can find that some edges are single while the others are double; see, e.g., the right graph in Figure \ref{fig:butterfly_graph}.

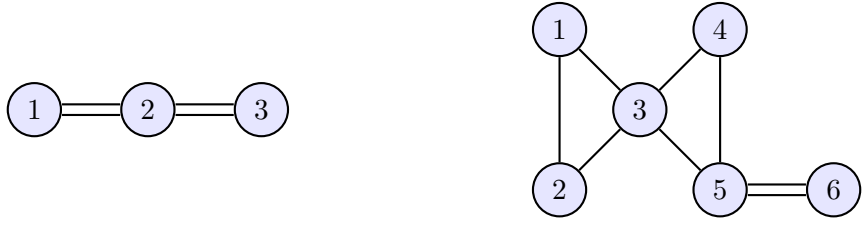
\begin{figure}[htbp]
    \centering
    \begin{tikzpicture}[
        every node/.style={circle, draw, fill=blue!10, minimum size=20pt},
        node distance=2cm,
        line width=0.8pt
    ]
		
			\begin{scope}[xshift=-8cm, node distance=1.5cm]
          \node (n1) at (0,0) {1};
          \node (n2) [right of=n1] {2};
          \node (n3) [right of=n2] {3};
          
          \draw ([yshift=2pt]n1.east) -- ([yshift=2pt]n2.west);
          \draw ([yshift=-2pt]n1.east) -- ([yshift=-2pt]n2.west);
          
          \draw ([yshift=2pt]n2.east) -- ([yshift=2pt]n3.west);
          \draw ([yshift=-2pt]n2.east) -- ([yshift=-2pt]n3.west);
      \end{scope}
    
			\begin{scope}[node distance=1.5cm]
          \node (3) at (0,0) {3};
          \node (1) [above left of=3] {1};
          \node (2) [below left of=3] {2};
          \node (4) [above right of=3] {4};
          \node (5) [below right of=3] {5};
          \node (6) [right of=5] {6};
        
          \draw (1) -- (2);
          \draw (1) -- (3);
          \draw (2) -- (3);
          \draw (3) -- (4);
          \draw (3) -- (5);
          \draw (4) -- (5);
          
          \draw ([yshift=2pt]5.east) -- ([yshift=2pt]6.west);
          \draw ([yshift=-2pt]5.east) -- ([yshift=-2pt]6.west);
      \end{scope}

    \end{tikzpicture}
    \caption{Left: A multicycle with a degree-$4$ node and two double edges. Right: A multicycle with two degree-$4$ nodes ($3$ and $5$) and a double edge.}
    \label{fig:butterfly_graph}
\end{figure}

For the proof of Proposition \ref{prop:main}, we need to consider even more complicated graphs, where four different types of random variables $P^{(n)}_{ij}$ for $n=1, 2, 3, 4$ in \eqref{eq:convergence_iid} or \eqref{eq:convergence_normalized} can be attached to each edge, and also other types of random variables $P^{(n)}_{d, kk}$ ($n=1, 2$) should be considered. This will be done by using the multigraphs.

\subsection{Definition of multigraphs} \label{subsec:multigraph}

We introduce a one-to-one correspondence between each term in the product in \eqref{eq:convergence_iid} or \eqref{eq:convergence_normalized} and a multigraph with $N$ nodes for which we will use the following notations:
\begin{defn}[$\ell$-multigraph] \label{def:l-graph}
We say that a multigraph $\gamma$ with nodes $1, 2, \dots, N$ is an $\ell$-multigraph if any pair of nodes can be connected by at most $\ell$ edges. The multiplicity $m^\gamma_{ij}$ between the nodes $i$ and $j$ is the number of separate edges connecting them, and we let $m^\gamma_{ij} = 0$ if the nodes $i$ and $j$ are not connected. (It is possible that $m_{kk}^{\gamma} > 0$, i.e., self-loops are allowed.) If $m^\gamma_{ij} = 1, 2, 3, 4$, we say that the edge between the nodes $i$ and $j$ is single, double, triple, quadruple, respectively.

The degree of the node $i$ is defined by
\[
	d^\gamma_i = m^\gamma_{ii} + \sum_{j=1}^N m^\gamma_{ij}.
\]
In particular, a single self-loop at the node $i$ increases the degree of the node $i$ by $2$. The size $m^\gamma$ and the length $n^\gamma$ of a multigraph $\gamma$ are defined by
\[
	m^\gamma = \sum_{i\leq j} m^\gamma_{ij} = \frac{1}{2} \sum_{i=1}^N d^\gamma_i, \quad n^\gamma = \sum_{i\leq j} \indi(m^\gamma_{ij} > 0),
\]
respectively. The number of the single (respectively, double, triple, quadruple) edges $n_1^\gamma$ (respectively, $n_2^\gamma, n_3^\gamma, n_4^\gamma$) is defined by
\[
	n_r^\gamma = \sum_{i\leq j} \indi(m^\gamma_{ij} = r), \quad r=1, 2, 3, 4.
\]
Similarly, the number of the single (respectively, double) self-loops $\tilde n_1^\gamma$ (respectively, $\tilde n_2^\gamma$) is defined by
\[ \tilde n_r^\gamma = \sum_{k=1}^{N}\indi(m^\gamma_{kk} = r), \quad r=1,2.\]

The set of all $\ell$-multigraphs with single and/or double self-loops is denoted by $\caG^\ell$. (The empty graph is a multigraph and included in $\caG^\ell$.)
\end{defn}

To demonstrate the multigraph expansion, we first recall the product in \eqref{eq:convergence_iid},
\[ \begin{split}
	&\prod_{i<j} \left( 1 + P^{(1)}_{ij} \sqrt{N} x_i x_j + P^{(2)}_{ij} N x_i^2 x_j^2 + P^{(3)}_{ij} N^{3/2} x_i^3 x_j^3 + P^{(4)}_{ij} N^2 x_i^4 x_j^4 \right) \\
	&\qquad \times \prod_{k=1}^N \left( 1 + P^{(1)}_{d, kk} \sqrt{N} x_k^2 + P^{(2)}_{d, kk} N x_k^4 \right).
\end{split} \]
We expand this product and rewrite it by using the $4$-multigraphs as
\[ \begin{split}
	& \sum_{\gamma \in \caG^4} \left( \prod_{i<j} P_{ij}^{(m^\gamma_{ij})} (\sqrt{N} x_i x_j)^{m^\gamma_{ij}} \right) \left( \prod_{k=1}^N P_{d, kk}^{(m^\gamma_{kk})} (\sqrt{N} x_k^2)^{m^\gamma_{kk}} \right) \\
	&= \sum_{\gamma \in \caG^4} N^{m^\gamma/2} \left( \prod_{i<j} P_{ij}^{(m^\gamma_{ij})} \right) \left( \prod_{k=1}^N P_{d, kk}^{(m^\gamma_{kk})} \right) \left( \prod_{k=1}^N x_k^{d^\gamma_k} \right),
\end{split} \]
where we define $P_{ij}^{(0)} = P_{d, kk}^{(0)} = 1$. A similar expansion is also possible for the product in \eqref{eq:convergence_normalized}.

Due to the symmetry of the prior, we find that if any of the nodes in $\gamma$ has an odd degree then $\E_\caX \left[ \prod_{k=1}^N x_k^{d^\gamma_k} \right]= 0$. Thus, we are led to consider $4$-multigraphs with even degrees for which we use the following definition:
\begin{defn}[$\ell$-multicycle] \label{def:l-cycle}
We say that a multigraph $\gamma$ is a multicycle if the degree of any node of $\gamma$ is even. The set of all $\ell$-multicycles is denoted by $\caG_L^\ell$, and the set of all $\ell$-multicycles without self-loops is denoted by $\caG_C^\ell$
\end{defn}

Note that the index $\ell$ in the multicycle (or the multigraph) denotes the maximum number of the parallel edges between the nodes of the multigraph, not the length of the multigraph as in some literature in graph theory.

From now on, we will exclusively consider the multicycles. More precisely, we decompose the prior by letting
\beq \label{eq:def_y}
	x_i = \frac{1}{\sqrt{N}} v_i y_i, \quad \text{ or } \quad \frac{x_i}{\| \bsx \|} = \frac{1}{\sqrt{N}} v_i y_i,
\eeq
depending on whether we consider \eqref{eq:convergence_iid} or \eqref{eq:convergence_normalized}, where $v_i$'s are independent Rademacher random variables and $y_i$'s are non-negative identically distributed random variables. Denote by $\caY$ the probability space for $\bsy=(y_1, \ldots, y_N)$. Since $v_i$'s are Rademacher random variables, by considering the expectation with respect to $v_i$ first, and then the expectation with respect to $y_i$'s, we are led to consider
\beq \label{eq:X_before_conditioning}
	\sum_{\gamma \in \caG_L^4} \frac{1}{N^{m^\gamma/2}} \E_\caY \left[ \prod_{k=1}^N y_k^{d^\gamma_k} \right] \left( \prod_{i<j} P_{ij}^{(m^\gamma_{ij})} \right) \left( \prod_{k=1}^N P_{d, kk}^{(m^\gamma_{kk})} \right).
\eeq

\subsection{Notations and assumptions for the multigraph expansion}

Instead of analyzing the object in \eqref{eq:X_before_conditioning} directly, we use the `conditioning', i.e., we assume a high probability event with respect to $\caX$ in the analysis of \eqref{eq:X_before_conditioning}. Since
\[
	\E \left[ \left( \prod_{i<j} P_{ij}^{(m^\gamma_{ij})} \right) \left( \prod_{k=1}^N P_{d, kk}^{(m^\gamma_{kk})} \right) \right] = 0
\]
for any $\gamma \in \caG_L^4$ except when $\gamma$ is empty, by applying Lemma \ref{lem:truncate}, we can see that any event with probability $o(1)$ is negligible in the proof of Proposition \ref{prop:main}. 

One of the most important features of the conditioning is that we need Assumption \ref{assump:second_moment} to control large graphs. In fact, we will strengthen Assumption \ref{assump:second_moment} as follows.

\begin{lem} \label{lem:Gaussian_second_moment}
Suppose that Assumption \ref{assump:second_moment} holds with a constant $\SNR_c$. Then, there exists a sequence of events $\Omega_N'$ in $\caX$ with $\p(\Omega_N'^c) = o(1)$ such that for any $\SNR < \SNR_c$,
\beq \label{eq:Gaussian_second_moment_1}
	\E_{\bsx, \bsx' \sim \caX} \left[ \indi(\Omega_N' (\bsx)) \indi(\Omega_N' (\bsx')) \exp \left( \frac{N \SNR \langle \bsx, \bsx' \rangle^2}{2} \right) \right] = (1-\SNR)^{-1/2} + o(1)
\eeq
and
\beq \label{eq:Gaussian_second_moment_2}
	\E_{\bsx, \bsx' \sim \caX} \left[ \indi(\Omega_N' (\bsx)) \indi(\Omega_N' (\bsx')) \exp \left( \frac{N \SNR \langle \bsx, \bsx' \rangle^2}{2 \| \bsx \|^2 \| \bsx' \|^2} \right) \right] = (1-\SNR)^{-1/2} + o(1).
\eeq
\end{lem}

We will prove Lemma \ref{lem:Gaussian_second_moment} in Section \ref{sec:proofs}. In the rest of the paper, we will assume that the following event holds:
\begin{defn} \label{defn:high-probability}
We assume that $\Omega = \Omega'_N \cap \Omega'' \cap \Omega_\caX^\epsilon$, where $\Omega'_N$ is the event in Lemma \ref{lem:Gaussian_second_moment}, $\Omega''$ in \eqref{eq:Omega''}, and $\Omega_\caX^\epsilon$ in Proposition \ref{prop:main}. Note that $\Omega$ is high probability with respect to $\caX$ such that
\begin{itemize}
	\item Lemma \ref{lem:Gaussian_second_moment} holds with $\Omega$,
	\item for the event $\Omega''$, $\Omega \subset \Omega''$, and
	\item the event $\Omega$ is symmetric in the sense that $\Omega$ holds with $(x_1, \dots, x_{i-1}, -x_i, x_{i+1}, \dots, x_N)$ if $\Omega$ holds with $(x_1, \dots, x_{i-1}, x_i, x_{i+1}, \dots, x_N)$ for any $1\le i \le N$.
\end{itemize}
\end{defn}

We now let
\[
	X := \sum_{\gamma \in \caG_L^4} \frac{1}{N^{m^\gamma/2}} \E_\caY \left[ \indi(\Omega) \prod_{k=1}^N y_k^{d^\gamma_k} \right] \left( \prod_{i<j} P_{ij}^{(m^\gamma_{ij})} \right) \left( \prod_{k=1}^N P_{d, kk}^{(m^\gamma_{kk})} \right)
\]
with the event $\Omega$ in Definition \ref{defn:high-probability}. Since $\Omega$ is measurable with respect to $\caY$,
\[
	X - \sum_{\gamma \in \caG_L^4} \frac{1}{N^{m^\gamma/2}} \E_\caY \left[ \prod_{k=1}^N y_k^{d^\gamma_k} \right] \left( \prod_{i<j} P_{ij}^{(m^\gamma_{ij})} \right) \left( \prod_{k=1}^N P_{d, kk}^{(m^\gamma_{kk})} \right) = o_\caW(1).
\]
It is clear that in order to prove Proposition \ref{prop:main}, it suffices to show that $\log X \Rightarrow \caN(-\rho, 2\rho)$.

Lastly, we introduce the following to ease the notation in the analysis of $X$.
\begin{defn} \label{def:prior_factor}
For a multicycle $\gamma$, we define its prior factor restricted on a good high probability event $Y(\gamma)$ by
\[
	Y(\gamma) := \E_\caY \left[ \indi(\Omega) \prod_{k=1}^N y_k^{d^\gamma_k} \right]
\]
and the random variable associated with the graph $\gamma$
\[
	Z(\gamma) := \left( \prod_{i<j} P_{ij}^{(m^\gamma_{ij})} \right) \left( \prod_{k=1}^N P_{d, kk}^{(m^\gamma_{kk})} \right).
\]
\end{defn}
With Definition \ref{def:prior_factor}, we can find
\beq \label{eq:X_SNR}
	X = \sum_{\gamma \in \caG_L^4} \frac{1}{N^{m^\gamma/2}} Y(\gamma) Z(\gamma).
\eeq
We remark that if $\gamma$ and $\gamma'$ do not intersect, i.e., $m_{ij}^{\gamma} m_{ij}^{\gamma'} = 0$ for any $i \leq j$, then $Z(\gamma \cup \gamma') = Z(\gamma) Z(\gamma')$.

\section{Proof of Proposition \ref{prop:main}} \label{sec:main_proof}

In this section, we prove Proposition \ref{prop:main}. 

\subsection{1-multicycle universality} \label{subsec:1-cycle}

In the first step of the proof of Proposition \ref{prop:main}, we focus on the contribution from the $1$-multicycles without self-loops to the sum $X$ in \eqref{eq:X_SNR}. Recall that $\caG_C^1$ is the set of all $1$-multicycles without self-loops. Our goal in this subsection is to prove the following lemma that we will call the $1$-multicycle universality.
\begin{lem} \label{lem:1-cycle_universality}
Define
\beq \label{eq:object_1}
	X_1 := \sum_{\gamma \in \caG_C^1} \frac{1}{N^{m^\gamma/2}} Y(\gamma) Z(\gamma)
\eeq
and
\beq \label{eq:object_R}
	R_1 := \sum_{\gamma \in \caG_C^1} \frac{1}{N^{m^\gamma/2}} Z(\gamma).
\eeq
Then, $\E[(X_1 - R_1)^2] = o(1)$.
\end{lem}

Before proving Lemma \ref{lem:1-cycle_universality}, we prove a result on the monotonicity of the variance of $X_1$ with respect to $\F$, which will be used later.
\begin{lem} \label{lem:monotonic_variance}
The variance $\E[X_1^2]$ depends only on $\F$ and $\caY$. If we consider $\E[X_1^2]$ as a function of $\F$, $V(\F) = \E[X_1^2]$, then for any $0 < \alpha_1 \leq \alpha_2$, we have $V(\alpha_1) \leq V(\alpha_2)$.
\end{lem}

\begin{proof}
We first expand $\E[X_1^2]$ as
\beq \begin{split} \label{eq:distinct_1-cycle}
	\E[X_1^2] 
	= \sum_{\gamma, \gamma' \in \caG_C^1} \frac{1}{N^{(m^\gamma + m^{\gamma'})/2}} Y(\gamma) Y(\gamma') \E \left[ Z(\gamma) Z(\gamma') \right].
\end{split} \eeq
We expect that the summand in the right side of \eqref{eq:distinct_1-cycle} vanishes unless the multicycles $\gamma$ and $\gamma'$ are the same. To consider such cases systematically, we introduce the concept of the orthogonality.
\begin{defn}[orthogonal multicycles] \label{def:orthogonal}
We say that two multicycles $\gamma$ and $\gamma'$ are orthogonal if
\beq \label{eq:orthogonal}
	\E[ Z(\gamma) Z(\gamma')] = \E \left[ \prod_{i<j} P_{ij}^{(m^\gamma_{ij})} \prod_{i'<j'} P_{i'j'}^{(m^{\gamma'}_{i'j'})} \right] = 0.
\eeq
\end{defn}
Note that the orthogonality depends not only on the multicycles $\gamma$ and $\gamma'$ but also $P_{ij}^{(n)}$.

Suppose that $\gamma$ and $\gamma'$ in $\caG_C^1$ are distinct. If we assume without loss of generality that $m^\gamma_{ab} = 1$ but $m^{\gamma'}_{ab} = 0$ for some nodes $a$ and $b$, we can pull out the factor $\E[P_{ab}^{(1)}]$ from the last term of the multiplicand in the right side of \eqref{eq:orthogonal}, which vanishes. Thus, distinct $1$-multicycles are orthogonal and we find that
\[
	\E[X_1^2] 
	= \sum_{\gamma \in \caG_C^1} \frac{1}{N^{m^\gamma}} Y(\gamma)^2 \E \left[ Z(\gamma)^2 \right] = \sum_{\gamma \in \caG_C^1} \left( \frac{\F}{N} \right)^{m^\gamma} Y(\gamma)^2,
\]
which shows that $\E[X_1^2]$ depends only on $\F$ and $\caY$, and also $\E[X_1^2]$ is an increasing function of $\F$.
\end{proof}

We now return to the proof of Lemma \ref{lem:1-cycle_universality}. 

\begin{proof}[Proof of Lemma \ref{lem:1-cycle_universality}]
Adapting the proof of Lemma \ref{lem:monotonic_variance}, we find
\[
	\E[(X_1 - R_1)^2] = \E \left[ \left( \sum_{\gamma \in \caG_C^1} \frac{1}{N^{m^\gamma/2}} (Y(\gamma)-1) Z(\gamma) \right)^2 \right] = \sum_{\gamma \in \caG_C^1} \left( \frac{\F}{N} \right)^{m^\gamma} (Y(\gamma)-1)^2.
\]
We further decompose it as
\beq \begin{split} \label{eq:simple_cycle_decomposition}
	\E[(X_1 - R_1)^2]&= \sum_{\gamma \in \caG_C^1, Y(\gamma) \leq 1} \left( \frac{\F}{N} \right)^{m^\gamma} \left( Y(\gamma) -1 \right)^2 + \sum_{\gamma \in \caG_C^1, Y(\gamma) > 1} \left( \frac{\F}{N} \right)^{m^\gamma} \left( Y(\gamma) -1 \right)^2 \\
	&\le \sum_{\gamma \in \caG_C^1, Y(\gamma) \leq 1} \left( \frac{\F}{N} \right)^{m^\gamma} \left(1- Y(\gamma)^2 \right) + \sum_{\gamma \in \caG_C^1, Y(\gamma) > 1} \left( \frac{\F}{N} \right)^{m^\gamma} \left( Y(\gamma)^2 -1 \right)
\end{split} \eeq
To estimate the first term in the right side of \eqref{eq:simple_cycle_decomposition}, we consider the cutoff property for Rademacher prior, which is a slightly modified version of Lemma 3.3 in \cite{AizenmanLebowitzRuelle}. 
\begin{lem}[1-multicycle cutoff property for Rademacher prior] \label{lem:ALR_cutoff}
If $\alpha < 1$, then there exists a constant $s_1$, independent of any other quantities, such that for any $s > s_1$
\[
	\sum_{\gamma \in \caG_C^1, n^\gamma \geq s} \left( \frac{\alpha}{N} \right)^{m^{\gamma}} \leq s \alpha^s.
\]
\end{lem}
We prove Lemma \ref{lem:ALR_cutoff} in Section \ref{sec:proofs}.
By Lemma \ref{lem:ALR_cutoff}, we notice that for any fixed $\epsilon > 0$, there exists ($N$-independent) constant $K_{\epsilon} \equiv K_\epsilon(\F)$ such that for any $s \geq K_\epsilon$ and for any sufficiently large $N$,
\beq \label{eq:Rademacher_cutoff}
	\sum_{\gamma \in \caG_C^1, n^\gamma \geq s} \left( \frac{\F}{N} \right)^{m^{\gamma}} <\epsilon.
\eeq
Since the summand in the left side of \eqref{eq:Rademacher_cutoff} is positive,
\[ \begin{split}
	\sum_{\gamma \in \caG_C^1, Y(\gamma) \leq 1, n^\gamma \geq K_\epsilon} \left( \frac{\F}{N} \right)^{m^\gamma} Y(\gamma)^2 \leq \sum_{\gamma \in \caG_C^1, Y(\gamma) \leq 1, n^\gamma \geq K_\epsilon} \left( \frac{\F}{N} \right)^{m^\gamma} \leq \sum_{\gamma \in \caG_C^1, n^\gamma \geq K_\epsilon} \left( \frac{\F}{N} \right)^{n^{\gamma}} <\epsilon,
\end{split} \]
and in particular,
\beq \label{eq:normalized_1-cycle_first_term_1}
	\sum_{\gamma \in \caG_C^1, Y(\gamma) \leq 1, n^\gamma \geq K_\epsilon} \left( \frac{\F}{N} \right)^{n^{\gamma}} \left( 1 - Y(\gamma)^2 \right) <\epsilon.
\eeq

To control the sum in the case $Y(\gamma) \leq 1, n^\gamma < K_{\epsilon}$, we apply the following lemma:
\begin{lem} \label{lem:Y_gamma_dichotomy}
Assume that $\caX$ is not a Rademacher prior. Fix an ($N$-independent) constant $s$. Then, for any multicycle $\gamma$ with $n^\gamma<s$,
\begin{itemize}
	\item if all nodes of $\gamma$ are with degree $2$, then $Y(\gamma) \leq 1$ and $1- Y(\gamma) = o(1)$, and
	\item if at least one node of $\gamma$ is with degree $4$ or higher, then $Y(\gamma) > 1$ for any sufficiently large $N$.
\end{itemize}
\end{lem}
We prove Lemma \ref{lem:Y_gamma_dichotomy} in Section \ref{sec:proofs}. From Lemma \ref{lem:Y_gamma_dichotomy}, we find that
\beq \label{eq:normalized_1-cycle_first_term_2}
	\sum_{\gamma \in \caG_C^1, Y(\gamma) \leq 1, n^\gamma < K_{\epsilon}} \left( \frac{\F}{N} \right)^{n^{\gamma}} \left( 1 - Y(\gamma)^2 \right) = o(1),
\eeq
since the number of $1$-multicycles with the length $n^{\gamma}$ whose all nodes are with degree $2$ is less than $N^{n^{\gamma}}$.

To estimate the second term in the right side of \eqref{eq:simple_cycle_decomposition}, we apply the following lemma:
\begin{lem} \label{lem:Rademacher_difference_positive}
If $\F <\SNR_c$,
\beq \label{eq:4-cycle_cutoff}
	\sum_{\gamma \in \caG_C^1, Y(\gamma)>1} \left( \frac{\F}{N} \right)^{n^{\gamma}} (Y(\gamma)^2-1) = o(1).
\eeq
\end{lem}
We prove Lemma \ref{lem:Rademacher_difference_positive} in Section \ref{sec:proofs}. Combining \eqref{eq:normalized_1-cycle_first_term_1}, \eqref{eq:normalized_1-cycle_first_term_2}, and \eqref{eq:4-cycle_cutoff}, we conclude that Lemma \ref{lem:1-cycle_universality} holds.
\end{proof}

\subsection{2-multicycle universality} \label{subsec:2-cycle}

As the second step, we prove the following estimate on $2$-multicycles without self-loops, which extends Lemma \ref{lem:1-cycle_universality}. 
\begin{lem} \label{lem:2-cycle_universality}
Recall the definition of $Y$ in Definition \ref{def:prior_factor}. Define
\beq \label{eq:object_2}
	X_2 := \sum_{\gamma \in \caG_C^2} \frac{1}{N^{m^\gamma/2}} Y (\gamma) Z(\gamma)
\eeq
and
\beq \label{eq:object_R2}
	R_2 := \sum_{\gamma \in \caG_C^2} \frac{1}{N^{m^\gamma/2}} Z(\gamma).
\eeq
Then, $\E[(X_2 - R_2)^2] = o(1)$.
\end{lem}

\begin{proof}
Recall that $n_1^\gamma$ is the number of the single edges in $\gamma$ and $n_2^\gamma$ the number of the double edges, and in particular, $\E[Z(\gamma)^2] = \F^{n_1^\gamma} \G^{n_2^\gamma}$. If $\gamma$ and $\gamma'$ in $\caG_C^2$ are distinct, then they are orthogonal since we are assuming $\E[P_{ij}^{(1)} P_{ij}^{(2)}] = 0$. Thus, from $m^\gamma = n_1^\gamma + 2n_2^\gamma$,
\beq \begin{split} \label{eq:2-cycle_decomposition}
	\E[(X_2 - R_2)^2] = \sum_{\gamma \in \caG_C^2} \left( \frac{1}{N} \right)^{m^\gamma} \F^{n_1^\gamma} \G^{n_2^\gamma} \left( Y(\gamma) -1 \right)^2 = \sum_{\gamma \in \caG_C^2} \left( \frac{\F}{N} \right)^{n_1^\gamma} \left( \frac{\G}{N^2} \right)^{n_2^\gamma} \left( Y(\gamma) -1 \right)^2.
\end{split} \eeq
We first prove the cutoff property for the right side of \eqref{eq:2-cycle_decomposition}, which asserts that in this sum the contribution from large graphs is negligible.
To estimate the contribution from large graphs, we will use the fact that $(Y(\gamma)-1)^2 \le Y(\gamma)^2 +1$.

For $\gamma \in \caG_C^2$, let $\gamma^+$ be the set of $2$-multicycles (without self-loops) that can be obtained by adding a double edge to $\gamma$. Then, since connecting nodes $i$ and $j$ by a double edge increases the degrees $d_i^{\gamma}$ and $d_j^{\gamma}$ by $2$, respectively, and remain other degrees unchanged,
\[ \begin{split}
	&\sum_{\gamma' \in \gamma^+} \left( \frac{\F}{N} \right)^{n_1^{\gamma'}} \left( \frac{\G}{N^2} \right)^{n_2^{\gamma'}} Y(\gamma')^2 \leq \left( \frac{\F}{N} \right)^{n_1^\gamma} \left( \frac{\G}{N^2} \right)^{n_2^\gamma +1} \sum_{i<j} \left(\E_\caY \left[ \indi(\Omega) y_i^2 y_j^2 \prod_{k=1}^N y_k^{d^{\gamma}_k} \right] \right)^2 \\
	&\leq \left( \frac{\F}{N} \right)^{n_1^\gamma} \left( \frac{\G}{N^2} \right)^{n_2^\gamma +1} \sum_{i<j} \left(\E_\caY \left[ \indi(\Omega) y_i^4 y_j^4 \prod_{k=1}^N y_k^{d^{\gamma}_k} \right] \right) \left(\E_\caY \left[\prod_{k=1}^N \indi(\Omega) y_k^{d^{\gamma}_k} \right] \right) \\
	&= \left( \frac{\F}{N} \right)^{n_1^\gamma} \left( \frac{\G}{N^2} \right)^{n_2^\gamma +1} \left(\E_\caY \left[\prod_{k=1}^N \indi(\Omega) y_k^{d^{\gamma}_k} \right] \right) \left(\E_\caY \left[ \indi(\Omega) \left( \sum_{i<j} y_i^4 y_j^4 \right) \prod_{k=1}^N y_k^{d^{\gamma}_k} \right] \right).
\end{split} \]

From the definition of $\Omega$, we have
\[ \begin{split}
	\E_\caY \left[ \indi(\Omega) \left( \sum_{i<j} y_i^4 y_j^4 \right) \prod_{k=1}^N y_k^{d^{\gamma}_k} \right] &\leq \E_\caY \left[ \frac{\indi(\Omega)}{2} \left( \sum_{i=1}^N y_i^4 \right)^2 \prod_{k=1}^N y_k^{d^{\gamma}_k} \right] \\
	&\leq \frac{N^2 m_4^2}{2} (1+N^{-1/5}) \E_\caY \left[ \indi(\Omega) \prod_{k=1}^N y_k^{d^{\gamma}_k} \right].
\end{split} \]
We now have
\beq \label{eq:double_edge_induction}
\sum_{\gamma' \in \gamma^+} \left( \frac{\F}{N} \right)^{n_1^{\gamma'}} \left( \frac{\G}{N^2} \right)^{n_2^{\gamma'}} Y(\gamma')^2 \leq (m_4^2 \G) \left( \frac{\F}{N} \right)^{n_1^\gamma} \left( \frac{\G}{N^2} \right)^{n_2^\gamma} Y(\gamma)^2.
\eeq

Applying \eqref{eq:double_edge_induction} inductively, if we let $\gamma^{+r}$ be the set of $2$-multicycles (without self-loops) that can be obtained by adding $r$ double edges to $\gamma$, then
\[
	\sum_{\gamma' \in \gamma^{+r}} \left( \frac{\F}{N} \right)^{n_1^{\gamma'}} \left( \frac{\G}{N^2} \right)^{n_2^{\gamma'}} Y(\gamma')^2 \leq \frac{\left(m_4^2 \G \right)^r}{r!} \left( \frac{\F}{N} \right)^{n_1^\gamma} \left( \frac{\G}{N^2} \right)^{n_2^\gamma} Y(\gamma)^2,
\]
where the factor $1/(r!)$ is due to the fact that $r$ double edges added in the inductive step were counted $(r!)$-times. Summing it over $r$, we get
\beq \begin{split}\label{eq:2-cycle_cutoff_3}
	&\sum_{\gamma \in \caG_C^2, n_1^\gamma \geq s} \left( \frac{\F}{N} \right)^{n_1^\gamma} \left( \frac{\G}{N^2} \right)^{n_2^\gamma} Y(\gamma)^2 \leq \sum_{\gamma' \in \caG_C^1, n^{\gamma'} \geq s} \left( \frac{\F}{N} \right)^{n^{\gamma'}} Y(\gamma')^2 \sum_{r=0}^{\infty} \frac{\left(m_4^2 \G \right)^r}{r!} \\
	&= e^{m_4^2 \G} \sum_{\gamma' \in \caG_C^1, n^{\gamma'} \geq s} \left( \frac{\F}{N} \right)^{n^{\gamma'}} Y(\gamma')^2.
\end{split} \eeq
To estimate the right side of \eqref{eq:2-cycle_cutoff_3}, we need a cutoff property that generalizes Lemma \ref{lem:ALR_cutoff}.
\begin{lem}[1-multicycle cutoff property for general prior] \label{lem:cutoff_general}
If $\F < \SNR_c$, then for any fixed $\epsilon > 0$, there exists ($N$-independent) constant $K_\epsilon' \equiv K_\epsilon'(\F)$ such that for any $s\geq K_\epsilon'$ and for any sufficiently large $N$,
\[
	\sum_{\gamma \in \caG_C^1, n^\gamma \geq s} \left( \frac{\F}{N} \right)^{n^{\gamma}} Y(\gamma)^2 <\epsilon.
\]
\end{lem}
We prove Lemma \ref{lem:cutoff_general} in Section \ref{sec:proofs}. By Lemma \ref{lem:cutoff_general}, we conclude from \eqref{eq:2-cycle_cutoff_3} that
\[ 
	 \sum_{\gamma \in \caG_C^2, n_1^\gamma \geq s} \left( \frac{\F}{N} \right)^{n_1^\gamma} \left( \frac{\G}{N^2} \right)^{n_2^\gamma} Y(\gamma)^2 < e^{m_4^2 \G} \epsilon,
\]
and by considering the Rademacher prior,
\[ 
	 \sum_{\gamma \in \caG_C^2, n_1^\gamma \geq s} \left( \frac{\F}{N} \right)^{n_1^\gamma} \left( \frac{\G}{N^2} \right)^{n_2^\gamma} < e^{m_4^2 \G} \epsilon.
\]
This shows that when we prove that the right side of \eqref{eq:2-cycle_decomposition} is $o(1)$, we only need to consider $2$-multicycles $\gamma$ such that $n_1^\gamma < K_{\epsilon}'$.

Similarly, we can find a constant $\tilde K_{\epsilon}$ such that for any $r_0 \geq \tilde K_{\epsilon}$,
\beq \begin{split} \label{eq:2-cycle_cutoff_2}
	&\sum_{\gamma \in \caG_C^2, n_1^\gamma < K_\epsilon', n_2^\gamma \geq r_0} \left( \frac{\F}{N} \right)^{n_1^\gamma} \left( \frac{\G}{N^2} \right)^{n_2^\gamma} Y(\gamma)^2 \\
	&\leq \left( \sum_{\gamma' \in \caG_C^1, n^{\gamma'} < K_\epsilon'} \left( \frac{\F}{N} \right)^{n^{\gamma'}} Y(\gamma')^2 \right) \left( \sum_{r=r_0}^{\infty} \frac{\left(m_4^2 \G \right)^r}{r!} \right) < \epsilon,
\end{split} \eeq
since in the right side of \eqref{eq:2-cycle_cutoff_2}, the first sum is $O(1)$ as $K_\epsilon'$ is a constant, and the second sum is a tail part of the Taylor series of the exponential function. 
Also, by considering the Rademacher prior,
\[
    \sum_{\gamma \in \caG_C^2, n_1^\gamma < K_\epsilon', n_2^\gamma \geq r_0} \left( \frac{\F}{N} \right)^{n_1^\gamma} \left( \frac{\G}{N^2} \right)^{n_2^\gamma} <\epsilon.
\]
This shows that when we prove that the right side of \eqref{eq:2-cycle_decomposition} is $o(1)$, we only need to consider $2$-multicycles $\gamma$ such that $n_1^\gamma < K_{\epsilon}$ and $n_2^\gamma < \tilde K_{\epsilon}$.

It remains to estimate
\beq \label{eq:2-cycle_bounded_target}
	\sum_{\gamma \in \caG_C^2, n_1^\gamma < K_\epsilon', n_2^\gamma < \tilde K_{\epsilon}} \left( \frac{\F}{N} \right)^{n_1^\gamma} \left( \frac{\G}{N^2} \right)^{n_2^\gamma} (Y(\gamma)^2 -1).
\eeq
Following the proof of Lemma \ref{lem:1-cycle_universality}, for $\gamma \in \caG_C^2$ with $n_1^\gamma < K_{\epsilon}'$ and $n_2^\gamma < \tilde K_{\epsilon}$, we separate the case $Y(\gamma) \leq 1$ and $Y(\gamma) > 1$ to find that the sum in \eqref{eq:2-cycle_bounded_target} is $o(1)$. For the former, we apply the first part of Lemma \ref{lem:Y_gamma_dichotomy}. For the latter, we apply the second part of Lemma \ref{lem:Y_gamma_dichotomy} to find that $\gamma$ has a node with degree $4$ or higher. Since a $1$-multicycle with $s_1$ edges has at most $s_1$ nodes with degree $2$ or higher, each double edge connects $2$ nodes, and at least one node is with degree $4$ or higher, there are at most $N^{s_1 + 2s_2 - 1}$ different ways of choosing the nodes of $\gamma$. Since the number of ways of connecting the chosen nodes accordingly is bounded by a constant (depending only on $s_1$ and $s_2$), we find that
\[
	\sum_{\gamma \in \caG_C^2, n_1^\gamma =s_1 < K_\epsilon', n_2^\gamma =s_2 < \tilde K_{\epsilon}, Y(\gamma) > 1} \left( \frac{\F}{N} \right)^{n_1^\gamma} \left( \frac{\G}{N^2} \right)^{n_2^\gamma} (Y(\gamma)^2 -1) = O(N^{-1}).
\]
Summing it over $s_1 = 1, 2, \dots, K_\epsilon'$ and $s_2 = 1, 2, \dots, \tilde K_{\epsilon}$, we finish the proof of the desired lemma.
\end{proof}

\subsection{3-multicycle estimate} \label{subsec:3-cycle}

Our goal in this subsection is to prove the following estimate on $3$-multicycles without self-loops.
\begin{lem} \label{lem:3-cycle_estimate}
Recall the definition of $R_2$ in \eqref{eq:object_R2}. Define
\beq \label{eq:object_3}
	X_3:= \sum_{\gamma \in \caG_C^3} \frac{1}{N^{m^\gamma/2}} Y(\gamma) Z(\gamma).
\eeq
Then, $\E[(X_3 - R_2)^2] = o(1)$.
\end{lem}

\begin{proof}[Proof of Lemma \ref{lem:3-cycle_estimate}]
Since
\[
	\E[(X_3 - R_2)^2] \leq 2\E[(X_3 - X_2)^2] + 2\E[(X_2 - R_2)^2],
\]
and $\E[(X_2 - R_2)^2] = o(1)$ from Lemma \ref{lem:2-cycle_universality}, it suffices to show that $\E[(X_3 - X_2)^2] = o(1)$. We consider the expansion
\beq \begin{split} \label{eq:distinct_3-cycle}
	X_3 - X_2 = \sum_{\gamma \in \caG_C^3 \setminus \caG_C^2} \frac{1}{N^{m^\gamma/2}} Y(\gamma) Z(\gamma).
\end{split} \eeq
For the second moment computation, we adapt the proof of the $1$-multicycle universality. We introduce the concept of the similarity between multicycles as follows.
\begin{defn}[Similarity] \label{def:similar}
We say that two multicycles $\gamma$ and $\gamma'$ are similar and write $\gamma \sim \gamma'$ if the following holds: for any nodes $i$ and $j$, $m^\gamma_{ij} > 0$ iff $m^{\gamma'}_{ij} > 0$, and the difference of the multiplicities $|m^\gamma_{ij} - m^{\gamma'}_{ij}|$ is even.
\end{defn}
In words, $\gamma$ and $\gamma'$ are similar if the underlying simple graphs for $\gamma$ and $\gamma'$ coincide with each other and the parities of the multiplicities of corresponding edges of $\gamma$ and $\gamma'$ also match; in particular, if $\gamma \sim \gamma'$, then the lengths of $\gamma$ and $\gamma'$ are the same.

Suppose that we compute the second moment $\E[(X_3 - X_2)^2]$ as in the proof of Lemma \ref{lem:monotonic_variance}. If two $3$-multicycles $\gamma, \gamma' \in \caG_C^3$ are not similar, $m^\gamma_{ab} \neq 0$ but $m^{\gamma'}_{ab} = 0$ (or $m^\gamma_{ab} = 0$ but $m^{\gamma'}_{ab} \neq 0$), or the parities of $m^\gamma_{ab}$ and $m^{\gamma'}_{ab}$ does not match for some nodes $a$ and $b$. Then, we can show that they are orthogonal by pulling out the factor $\E[P_{ab}^{(i)}]$ ($i=1, 2,$ or $3$) or $\E[P_{ab}^{(i)}P_{ab}^{(2)}]$ ($i=1$ or $3$) in the definition of the orthogonality. This suggests that we may consider the equivalence classes of $3$-multicycles based on the similarity. Each equivalence class contains only one $2$-multicycle, so by using $2$-multicycles as representatives of the equivalence classes, we get
\[ \begin{split}
	&\E[(X_3 - X_2)^2] = \sum_{\gamma \in \caG_C^2} \E \left[ \left( \sum_{\gamma' \in \caG_C^3 \setminus \caG_C^2, \gamma' \sim \gamma} \frac{1}{N^{m^{\gamma'}/2}} Y(\gamma') Z(\gamma') \right)^2 \right] \\
	&= \sum_{\gamma \in \caG_C^2} \E \left[ \sum_{\gamma_1 \in \caG_C^3 \setminus \caG_C^2, \gamma_1 \sim \gamma} \frac{1}{N^{m^{\gamma_1}/2}} Y(\gamma_1) Z(\gamma_1) \sum_{\gamma_2 \in \caG_C^3 \setminus \caG_C^2, \gamma_2 \sim \gamma} \frac{1}{N^{m^{\gamma_2}/2}} Y(\gamma_2) Z(\gamma_2) \right].
\end{split} \]

To prove an upper bound for the right side of the equation above, we consider $\gamma_1, \gamma_2 \in \caG_C^3 \setminus \caG_C^2$ that are similar to $\gamma$. We want to estimate the sum of
\beq \begin{split} \label{eq:similar_3_cycle_comparison}
	&\E \left[ \left( \frac{1}{N^{m^{\gamma_1}/2}} Y(\gamma_1) Z(\gamma_1) \right) \left( \frac{1}{N^{m^{\gamma_2}/2}} Y(\gamma_2) Z(\gamma_2) \right) \right] \\
	&= \frac{1}{N^{(m^{\gamma_1} + m^{\gamma_2})/2}} Y(\gamma_1) Y(\gamma_2) \prod_{i<j} \E \left[ P_{ij}^{(m^{\gamma_1}_{ij})} P_{ij}^{(m^{\gamma_2}_{ij})} \right]
\end{split} \eeq
for all such $\gamma_1$ and $\gamma_2$. Applying the Schwarz inequality to that last factor in the right side of \eqref{eq:similar_3_cycle_comparison},
\[
	\E \left[ P_{ij}^{(m^{\gamma_1}_{ij})} P_{ij}^{(m^{\gamma_2}_{ij})} \right] \leq \E \left[ \left( P_{ij}^{(m^{\gamma_1}_{ij})} \right)^2 \right]^{1/2} \E \left[ \left( P_{ij}^{(m^{\gamma_2}_{ij})} \right)^2 \right]^{1/2},
\]
which gives a bound
\beq \begin{split} \label{eq:similar_3_cycle_comparison_2}
	&\left| \E \left[ \left( \frac{1}{N^{m^{\gamma_1}/2}} Y(\gamma_1) Z(\gamma_1) \right) \left( \frac{1}{N^{m^{\gamma_2}/2}} Y(\gamma_2) Z(\gamma_2) \right) \right] \right| \\
	&\leq \frac{1}{N^{(m^{\gamma_1} + m^{\gamma_2})/2}} Y(\gamma_1) Y(\gamma_2) \prod_{i<j} \left( \E \left[ \left( P_{ij}^{(m^{\gamma_1}_{ij})} \right)^2 \right]^{1/2} \E \left[ \left( P_{ij}^{(m^{\gamma_2}_{ij})} \right)^2 \right]^{1/2} \right).
\end{split} \eeq

We focus on the factors depending only on $\gamma_1$ from the right side of \eqref{eq:similar_3_cycle_comparison_2}, i.e.,
\[
	\frac{1}{N^{m^{\gamma_1}/2}} Y(\gamma_1) \prod_{i<j} \E \left[ \left( P_{ij}^{(m^{\gamma_1}_{ij})} \right)^2 \right]^{1/2},
\]
and estimate them in terms of $\gamma$. Recall the definitions of the size $m^{\gamma}$ and the length $n^{\gamma}$ of a multicycle in Definition \ref{def:l-graph}. We have
\[
	\prod_{i<j} \E \left[ \left( P_{ij}^{(m^{\gamma_1}_{ij})} \right)^2 \right]^{1/2} = \F^{n_1^{\gamma_1}/2} \G^{n_2^{\gamma_1}/2} \cn_3^{n_3^{\gamma_1}/2}.
\]
If we compare the degrees of $\gamma_1$ and $\gamma$, we can easily see that $d^{\gamma_1}_k \geq d^{\gamma}_k$ for any node $k$. From the definition of $\Omega$, we know that $y_k \in [0, K]$ on $\Omega$ for $K=2N^{1/60}$. Thus, we also have $y_k^{d^{\gamma_1}_k} \leq K^{d^{\gamma_1}_k - d^{\gamma}_k} y_k^{d^{\gamma}_k}$. Now, from the relation $m^{\gamma_1} - m^{\gamma} = 2n_3^{\gamma_1}$,
\[ \begin{split}
	Y(\gamma_1) = \E_\caY \left[ \indi(\Omega) \prod_{k=1}^N y_k^{d^{\gamma_1}_k} \right] &\leq \E_\caY \left[ \indi(\Omega) \prod_{k=1}^N K^{d^{\gamma_1}_k - d^{\gamma}_k} y_k^{d^{\gamma}_k} \right] \\
	&= K^{\sum_k d^{\gamma_1}_k - \sum_k d^{\gamma}_k} \E_\caY \left[ \indi(\Omega) \prod_{k=1}^N y_k^{d^{\gamma}_k} \right] = K^{4n_3^{\gamma_1}} Y(\gamma).
\end{split} \]
Since $\gamma_1 \sim \gamma$, we find that $n_1^{\gamma} = n_1^{\gamma_1} + n_3^{\gamma_1}$ and $n_2^{\gamma_1} = n_2^{\gamma}$. So far, we have seen
\[ \begin{split}
	&\frac{1}{N^{m^{\gamma_1}/2}} Y(\gamma_1) \prod_{i<j} \E \left[ \left( P_{ij}^{(m^{\gamma_1}_{ij})} \right)^2 \right]^{1/2} = \frac{1}{N^{(n_1^{\gamma_1} + 2n_2^{\gamma_1} + 3n_3^{\gamma_1})/2}} Y(\gamma_1) \F^{n_1^{\gamma_1}/2} \G^{n_2^{\gamma_1}/2} \cn_3^{n_3^{\gamma_1}/2} \\
	&\leq \left( \frac{\F}{N} \right)^{n_1^{\gamma}/2} \left( \frac{\G}{N^2} \right)^{n_2^{\gamma}/2} Y(\gamma) \left( \frac{K^4 \sqrt{\cn_3}}{N\sqrt{\F}} \right)^{n_3^{\gamma_1}}.
\end{split} \]

For a fixed $s \in [[1, n_1^{\gamma}]]$, there are exactly $\binom{n_1^{\gamma}}{s}$ odd multicycles that are similar to $\gamma$ and have $s$ triple edges. Thus, we find from the binomial theorem that
\beq \begin{split}
	&\E \left[ \left( \sum_{\gamma' \in \caG_C^3 \setminus \caG_C^2, \gamma' \sim \gamma} \frac{1}{N^{m^{\gamma'}/2}} Y(\gamma') Z(\gamma') \right)^2 \right] \\
	&\leq \left( \left( \frac{\F}{N} \right)^{n_1^{\gamma}/2} \left( \frac{\G}{N^2} \right)^{n_2^{\gamma}/2} Y(\gamma) \sum_{s=1}^{n_1^{\gamma}} \binom{n_1^{\gamma}}{s} \left( \frac{K^4 \sqrt{\cn_3}}{N\sqrt{\F}} \right)^s \right)^2 \\ 
	&= \left( \frac{\F}{N} \right)^{n_1^{\gamma}} \left( \frac{\G}{N^2} \right)^{n_2^{\gamma}} Y(\gamma)^2 \left[ \left( 1 + \frac{K^4 \sqrt{\cn_3}}{N\sqrt{\F}} \right)^{n_1^\gamma} - 1 \right]^2.
\end{split} \eeq
Summing over all $2$-multicycle $\gamma$, we now have
\beq \begin{split} \label{eq:3-cycle_target}
	\E[(X_3 - X_2)^2] &\leq \sum_{\gamma \in \caG_C^2} \left( \frac{\F}{N} \right)^{n_1^{\gamma}} \left( \frac{\G}{N^2} \right)^{n_2^{\gamma}} Y(\gamma)^2 \left[ \left( 1 + \frac{K^4 \sqrt{\cn_3}}{N\sqrt{\F}} \right)^{n_1^\gamma} - 1 \right]^2 \\
	&\leq \sum_{\gamma \in \caG_C^2} \left( \frac{\F}{N} \right)^{n_1^{\gamma}} \left( \frac{\G}{N^2} \right)^{n_2^{\gamma}} Y(\gamma)^2 \left[ \left( 1 + \frac{K^4 \sqrt{\cn_3}}{N\sqrt{\F}} \right)^{2n_1^\gamma} - 1 \right].
\end{split} \eeq
To prove that the right side of the equation above is $o(1)$, we use the cutoff properties, Lemma \ref{lem:cutoff_general} and Equation \eqref{eq:2-cycle_cutoff_2}. Recall that we defined $K'_{\epsilon} \equiv K'_\epsilon(\F)$ in the proof of Lemma \ref{lem:cutoff_general}. Letting
\[
	\F'' := \F \left( 1 + \frac{K^4 \sqrt{\cn_3}}{N\sqrt{\F}} \right)^2,
\]
we find
\[ \begin{split}
	&\sum_{\gamma \in \caG_C^2, n_1^\gamma \geq K'_\epsilon(\F'')} \left( \frac{\F}{N} \right)^{n_1^{\gamma}} \left( \frac{\G}{N^2} \right)^{n_2^{\gamma}}Y(\gamma)^2 \left( 1 + \frac{K^4 \sqrt{\cn_3}}{N\sqrt{\F}} \right)^{2n_1^\gamma} \\
	&< \sum_{\gamma \in \caG_C^2, n_1^\gamma \geq K'_\epsilon(\F'')} \left( \frac{\F''}{N} \right)^{n_1^{\gamma}} \left( \frac{\G}{N^2} \right)^{n_2^{\gamma}} Y(\gamma)^2 <\epsilon,
\end{split} \]
since $\F < \F'' < \SNR_c$ for any sufficiently large $N$. Set $K''_{\epsilon} \equiv K'_\epsilon(\F'')$. Then,
\[ \begin{split}
	&\sum_{\gamma \in \caG_C^2} \left( \frac{\F}{N} \right)^{n_1^{\gamma}} \left( \frac{\G}{N^2} \right)^{n_2^{\gamma}} Y(\gamma)^2 \left[ \left( 1 + \frac{K^4 \sqrt{\cn_3}}{N\sqrt{\F}} \right)^{2n_1^\gamma} - 1 \right] \\
	&\leq \sum_{\gamma \in \caG_C^2, n_1^\gamma < K''_{\epsilon}} \left( \frac{\F}{N} \right)^{n_1^{\gamma}} \left( \frac{\G}{N^2} \right)^{n_2^{\gamma}} Y(\gamma)^2 \left[ \left( 1 + \frac{K^4 \sqrt{\cn_3}}{N\sqrt{\F}} \right)^{2n_1^\gamma} - 1 \right] + \epsilon \\
	&\leq \frac{C K^4}{N} \sum_{\gamma \in \caG_C^2, n_1^\gamma < K''_{\epsilon}} \left( \frac{\F}{N} \right)^{n_1^{\gamma}} \left( \frac{\G}{N^2} \right)^{n_2^{\gamma}} Y(\gamma)^2 + \epsilon \\
	&\leq \frac{C K^4}{N} \sum_{\gamma \in \caG_C^2, n_1^\gamma < K''_{\epsilon}, n_2^\gamma < \tilde K_{\epsilon}} \left( \frac{\F}{N} \right)^{n_1^{\gamma}} \left( \frac{\G}{N^2} \right)^{n_2^{\gamma}} Y(\gamma)^2 + 2\epsilon
\end{split} \]
for some constant $C$. In the proof of Lemma \ref{lem:2-cycle_universality}, we find that the right side in the equation above can be estimated as
\[
	\sum_{\gamma \in \caG_C^2, n_1^\gamma < K''_{\epsilon}, n_2^\gamma < \tilde K_{\epsilon}} \left( \frac{\F}{N} \right)^{n_1^{\gamma}} \left( \frac{\G}{N^2} \right)^{n_2^{\gamma}} Y(\gamma)^2 = \sum_{\gamma \in \caG_C^2, n_1^\gamma < K''_{\epsilon}, n_2^\gamma < \tilde K_{\epsilon}} \left( \frac{\F}{N} \right)^{n_1^{\gamma}} \left( \frac{\G}{N^2} \right)^{n_2^{\gamma}} + o(1),
\]
and we find by naive power counting that the right side is $O(1)$. Thus, we find that the right side of \eqref{eq:3-cycle_target} is bounded by $2\epsilon + O(N^{-14 / 15})$, and this completes the proof of the desired lemma.
\end{proof}

\subsection{4-multicycle estimate} \label{subsec:4-cycle}

We next consider the following estimate on $4$-multicycles without self-loops.
\begin{lem} \label{lem:4-cycle_estimate}
Recall the definition of $R_2$ in \eqref{eq:object_R2}. Define
\beq \label{eq:object_4}
	X_4 := \sum_{\gamma \in \caG_C^4} \frac{1}{N^{m^\gamma/2}} Y(\gamma) Z(\gamma).
\eeq
Then, 
\beq
	\E[(X_4 - R_2)^2] = o(1).
\eeq
\end{lem}

\begin{rem}
From Lemma \ref{lem:4-cycle_estimate}, we also find that $\E[(X_4 - R_4)^2] = o(1)$, where $R_4$ is defined by the right side of \eqref{eq:object_4} with $Y(\gamma) = 1$.
\end{rem}

\begin{proof}[Proof of Lemma \ref{lem:4-cycle_estimate}]
Since we closely follow the proof of Lemma \ref{lem:3-cycle_estimate}, we only outline the proof here. First, since
\[
	\E[(X_4 - R_2)^2] \leq 2\E[(X_4 - X_2)^2] + 2\E[(X_2 - R_2)^2],
\]
and $\E[(X_2 - R_2)^2] = o(1)$ from \ref{lem:2-cycle_universality}, it suffices to show that $\E[(X_4 - X_2)^2] = o(1)$. As in the proof of Lemma \ref{lem:3-cycle_estimate}, we consider the expansion
\beq \begin{split} \label{eq:distinct_4-cycle}
	X_4 - X_2 = \sum_{\gamma \in \caG_C^4 \setminus \caG_C^2} \frac{1}{N^{m^\gamma/2}} Y(\gamma) Z(\gamma)
\end{split} \eeq
and rewrite it by using the equivalence classes defined through the similarity in Definition \ref{def:similar}. We need to consider $\gamma_1, \gamma_2 \in \caG_C^4 \setminus \caG_C^2$ that are similar to $\gamma \in \caG_C^2$, and need to estimate the sum of the terms in \eqref{eq:similar_3_cycle_comparison_2}. We focus on the factors depending only on $\gamma_1$ from the right side of \eqref{eq:similar_3_cycle_comparison_2}, with
\[
	\prod_{i<j} \E \left[ \left( P_{ij}^{(m^{\gamma_1}_{ij})} \right)^2 \right]^{1/2} = \F^{n_1^{\gamma_1}/2} \G^{n_2^{\gamma_1}/2} \cn_3^{n_3^{\gamma_1}/2} \cn_4^{n_4^{\gamma_1}/2}.
\]
Recall that $y_k \in [0, K]$ on $\Omega$ for $K=2N^{1/60}$. We then have
\[ \begin{split}
	&\frac{1}{N^{m^{\gamma_1}/2}} Y(\gamma_1) \prod_{i<j} \E \left[ \left( P_{ij}^{(m^{\gamma_1}_{ij})} \right)^2 \right]^{1/2} \\
	&\leq \left( \frac{\F}{N} \right)^{n_1^{\gamma}/2} \left( \frac{\G}{N^2} \right)^{n_2^{\gamma}/2} Y(\gamma) \left( \frac{K^4 \sqrt{\cn_3}}{N\sqrt{\F}} \right)^{n_3^{\gamma_1}} \left( \frac{K^4 \sqrt{\cn_4}}{N\sqrt{\G}} \right)^{n_4^{\gamma_1}}.
\end{split} \]
From the binomial theorem,
\beq \begin{split}
	&\E \left[ \left( \sum_{\gamma' \in \caG_C^4 \setminus \caG_C^2, \gamma' \sim \gamma} \frac{1}{N^{m^{\gamma'}/2}} Y(\gamma') Z(\gamma') \right)^2 \right] \\
	&\leq \left( \left( \frac{\F}{N} \right)^{n_1^{\gamma}/2} \left( \frac{\G}{N^2} \right)^{n_2^{\gamma}/2} Y(\gamma) \left(\sum_{s_1=0}^{n_1^{\gamma}} \binom{n_1^{\gamma}}{s_1} \left( \frac{K^4 \sqrt{\cn_3}}{N\sqrt{\F}} \right)^{s_1} \sum_{s_2=0}^{n_2^{\gamma}} \binom{n_2^{\gamma}}{s_2} \left( \frac{K^4 \sqrt{\cn_4}}{N\sqrt{\G}} \right)^{s_2}-1\right)  \right)^2 \\ 
	&= \left( \frac{\F}{N} \right)^{n_1^{\gamma}} \left( \frac{\G}{N^2} \right)^{n_2^{\gamma}} Y(\gamma)^2 \left[ \left( 1 + \frac{K^4 \sqrt{\cn_3}}{N\sqrt{\F}} \right)^{n_1^\gamma}\left( 1 + \frac{K^4 \sqrt{\cn_4}}{N\sqrt{\G}} \right)^{n_2^\gamma} - 1 \right]^2 .
\end{split} \eeq
Now, applying the cutoff properties, Lemma \ref{lem:cutoff_general} and Equation \eqref{eq:2-cycle_cutoff_2}, we can complete the proof of the desired lemma.
\end{proof}

\subsection{Self-loop universality} \label{subsec:self-loop}

Recall that our main object $X$ is defined as the sum over $\gamma \in \caG_L^4$, the set of all $4$-multicycles (possibly with self-loops) and satisfies
\[
	X = \sum_{\gamma \in \caG_L^4} \frac{1}{N^{m^\gamma/2}} Y(\gamma) Z(\gamma) + o_\caW(1).
\]
We will prove the following result.

\begin{lem} \label{lem:self-loop_universality}
Define
\beq \label{eq:object_R2_tilde}
	\tilde R_2 := \sum_{\gamma \in \caG_L^2} \frac{1}{N^{m^\gamma/2}} Z(\gamma)
\eeq
Then, $\E[(X - \tilde R_2)^2] = o(1)$.
\end{lem}

\begin{proof}[Proof of Lemma \ref{lem:self-loop_universality}]
We follow the proof of Lemma \ref{lem:2-cycle_universality}. Let
\[
	\tilde X_2 := \sum_{\gamma \in \caG_L^2} \frac{1}{N^{m^\gamma/2}} Y(\gamma) Z(\gamma).
\]
We first want to show that
\beq \label{eq:tilde_X_2_estimate}
	\E[(\tilde X_2 - \tilde R_2)^2] = o(1).
\eeq
Notice that
\[
	\tilde X_2 - \tilde R_2 = \sum_{\gamma \in \caG_L^2} \frac{1}{N^{m^\gamma/2}} \left( Y(\gamma) -1 \right) Z(\gamma).
\]
As in the proof of Lemma \ref{lem:2-cycle_universality}, if $\gamma$ and $\gamma'$ in $\caG_L^2$ are distinct, then they are orthogonal. Thus,
\[ \begin{split}
	&\E[(\tilde X_2 - \tilde R_2)^2] = \sum_{\gamma \in \caG_L^2} \frac{1}{N^{m^\gamma}} \F^{n_1^\gamma - \tilde n_1^\gamma} \G^{n_2^\gamma - \tilde n_2^\gamma} (\cn_1')^{\tilde n_1^\gamma} (\cn_2')^{\tilde n_2^\gamma} \left( Y(\gamma) -1 \right)^2 \\
	&= \sum_{\gamma \in \caG_L^2} \left( \frac{\F}{N} \right)^{n_1^\gamma - \tilde n_1^\gamma} \left( \frac{\G}{N^2} \right)^{n_2^\gamma - \tilde n_2^\gamma} \left( \frac{\cn_1'}{N} \right)^{\tilde n_1^\gamma} \left( \frac{\cn_2'}{N^2} \right)^{\tilde n_2^\gamma} \left( Y(\gamma) -1 \right)^2.
\end{split} \]

We adapt the argument we used to derive \eqref{eq:2-cycle_decomposition}. Notice that adding a simple self-loop at node $i$ increases the degree $d_i^\gamma$ by $2$ and a double self-loop increases the degree by $4$. For $\gamma \in \caG_L^2$, let $\gamma^+$ be the set of $2$-multicycles (with self-loops) that can be obtained by adding a simple self-loop to $\gamma$. Then,
\[ \begin{split}
	&\sum_{\gamma' \in \gamma^+} \left( \frac{\F}{N} \right)^{n_1^{\gamma'} - \tilde n_1^{\gamma'}} \left( \frac{\G}{N^2} \right)^{n_2^{\gamma'} - \tilde n_2^{\gamma'}} \left( \frac{\cn_1'}{N} \right)^{\tilde n_1^{\gamma'}} \left( \frac{\cn_2'}{N^2} \right)^{\tilde n_2^{\gamma'}} Y(\gamma)^2 \\
	&\leq \left( \frac{\F}{N} \right)^{n_1^\gamma - \tilde n_1^\gamma} \left( \frac{\G}{N^2} \right)^{n_2^\gamma - \tilde n_2^\gamma} \left( \frac{\cn_1'}{N} \right)^{\tilde n_1^\gamma +1} \left( \frac{\cn_2'}{N^2} \right)^{\tilde n_2^\gamma} \sum_{\ell=1}^N \left(\E_\caY \left[ \indi(\Omega) y_\ell^2 \prod_{k=1}^N y_k^{d^{\gamma}_k} \right] \right)^2 \\
	&\leq \left( \frac{\F}{N} \right)^{n_1^\gamma - \tilde n_1^\gamma} \left( \frac{\G}{N^2} \right)^{n_2^\gamma - \tilde n_2^\gamma} \left( \frac{\cn_1'}{N} \right)^{\tilde n_1^\gamma +1} \left( \frac{\cn_2'}{N^2} \right)^{\tilde n_2^\gamma} (N m_4 +N^{4/5}) Y(\gamma)^2.
\end{split} \]
Similarly, for $\gamma \in \caG_L^2$, let $\gamma^{++}$ be the set of $2$-multicycles (with self-loops) that can be obtained by adding a double self-loop to $\gamma$. Then,
\[ \begin{split}
	&\sum_{\gamma' \in \gamma^{++}} \left( \frac{\F}{N} \right)^{n_1^{\gamma'} - \tilde n_1^{\gamma'}} \left( \frac{\G}{N^2} \right)^{n_2^{\gamma'} - \tilde n_2^{\gamma'}} \left( \frac{\cn_1'}{N} \right)^{\tilde n_1^{\gamma'}} \left( \frac{\cn_2'}{N^2} \right)^{\tilde n_2^{\gamma'}} Y(\gamma)^2 \\
	&\leq \left( \frac{\F}{N} \right)^{n_1^\gamma - \tilde n_1^\gamma} \left( \frac{\G}{N^2} \right)^{n_2^\gamma - \tilde n_2^\gamma} \left( \frac{\cn_1'}{N} \right)^{\tilde n_1^\gamma} \left( \frac{\cn_2'}{N^2} \right)^{\tilde n_2^\gamma +1} \sum_{\ell=1}^N \left(\E_\caY \left[ \indi(\Omega) y_\ell^4 \prod_{k=1}^N y_k^{d^{\gamma}_k} \right] \right)^2 \\
	&\leq \left( \frac{\F}{N} \right)^{n_1^\gamma - \tilde n_1^\gamma} \left( \frac{\G}{N^2} \right)^{n_2^\gamma - \tilde n_2^\gamma} \left( \frac{\cn_1'}{N} \right)^{\tilde n_1^\gamma} \left( \frac{\cn_2'}{N^2} \right)^{\tilde n_2^\gamma +1} (N m_8 +N^{4/5}) Y(\gamma)^2.
\end{split} \]

Let the $2$-submulticycle $\gamma'$ be the $2$-multicycle obtained by removing all self-loops in $\gamma$. By grouping the multicycles in $\caG_L^2$ according to their $2$-submulticycles, for any fixed $\epsilon > 0$ and for any $s \geq K_{\epsilon}$, 
\[ \begin{split}
	&\sum_{\gamma \in \caG_L^2, n_1^\gamma - \tilde n_1^\gamma > s} \left( \frac{\F}{N} \right)^{n_1^\gamma - \tilde n_1^\gamma} \left( \frac{\G}{N^2} \right)^{n_2^\gamma - \tilde n_2^\gamma} \left( \frac{\cn_1'}{N} \right)^{\tilde n_1^\gamma} \left( \frac{\cn_2'}{N^2} \right)^{\tilde n_2^\gamma} Y(\gamma)^2 \\
	&\leq \sum_{\gamma' \in \caG_C^2, n_1^{\gamma'} > s} \left( \frac{\F}{N} \right)^{n_1^{\gamma'}} \left( \frac{\G}{N^2} \right)^{n_2^{\gamma'}} Y(\gamma')^2 \sum_{r_1, r_2=0}^{\infty} \frac{1}{r_1! r_2!} (2 m_4 \cn_1')^{r_1} \left( \frac{2 m_8 \cn_2'}{N} \right)^{r_2} \\
	&\leq e^{2m_4\cn_1'} e^{2m_8\cn_2'/N} \sum_{\gamma' \in \caG_C^2, n_1^{\gamma'} > s} \left( \frac{\F}{N} \right)^{n_1^{\gamma'}} \left( \frac{\G}{N^2} \right)^{n_2^{\gamma'}} Y(\gamma')^2 \leq e^{2 m_4 \cn_1'} e^{2 m_8 \cn_2'/N} \epsilon.
\end{split} \]
This shows that any $2$-multicycle with $n_1^\gamma - \tilde n_1^\gamma \geq K_{\epsilon}$ is negligible when we prove \eqref{eq:tilde_X_2_estimate}. Similarly, we can also show that any $2$-multicycle with $n_2^\gamma - \tilde n_2^\gamma \geq \tilde K_{\epsilon}$ is negligible.

It remains to estimate
\[
	\sum_{\gamma \in \caG_L^2, n_1^\gamma - \tilde n_1^\gamma \leq K_{\epsilon}, n_2^\gamma - \tilde n_2^\gamma \leq \tilde K_{\epsilon}} \left( \frac{\F}{N} \right)^{n_1^\gamma - \tilde n_1^\gamma} \left( \frac{\G}{N^2} \right)^{n_2^\gamma - \tilde n_2^\gamma} \left( \frac{\cn_1'}{N} \right)^{\tilde n_1^\gamma} \left( \frac{\cn_2'}{N^2} \right)^{\tilde n_2^\gamma} (Y(\gamma)^2 -1).
\]
Following the proof of Lemma \ref{lem:1-cycle_universality} and Lemma \ref{lem:2-cycle_universality}, we separate the cases $Y(\gamma) \leq 1$ and $Y(\gamma) > 1$. Then, by applying both parts of Lemma \ref{lem:Y_gamma_dichotomy} for the cases, respectively, we can show that the equation above is $o(1)$. This proves \eqref{eq:tilde_X_2_estimate}.

To finish the proof of the desired lemma, we need to show that $\E[(X - \tilde X_2)^2] = o(1)$. It can be proved by adapting the proofs of Lemmas \ref{lem:3-cycle_estimate} and \ref{lem:4-cycle_estimate}. We omit the detail.
\end{proof}

\subsection{Decomposition for the Rademacher case} \label{subsec:Rademacher}

From Lemma \ref{lem:self-loop_universality}, $\E[(X - \tilde R_2)^2] = o(1)$, and thus we find that the limiting distribution of $X$ does not depend on the prior. More precisely, if we let
\beq \label{eq:R_SNR}
	R = \sum_{\gamma \in \caG_L^4} \frac{1}{N^{m^\gamma/2}} Z(\gamma),
\eeq
then
\[
	\E[(X - R)^2] \leq 2\E[(X - \tilde R_2)^2] + 2\E[(R - \tilde R_2)^2] = o(1).
\]

We now prove Proposition \ref{prop:main} by adapting the argument in \cite{AizenmanLebowitzRuelle}. Our goal is to show that $X$ can be approximated by a product involving simple cycles, double edges, and (single) self-loops. Our definition of the simple cycle is as follows:
\begin{defn}[simple cycle] \label{def:simple_cycle}
We say that a multicycle $\gamma$ is a simple cycle if $\gamma$ does not contain any self-loops, the degree of any node of $\gamma$ is at most $2$, and any two nodes with the degree $2$ are connected by edges. The set of all non-empty simple cycles is denoted by $\caG_S$.
\end{defn}
Recall the definition of $\tilde R_2$ in Lemma \ref{lem:self-loop_universality}. We have the following claim:
\beq \begin{split} \label{eq:claim_decomposition}
	\tilde R_2- \prod_{\gamma \in \caG_S} \left( 1+ \frac{Z(\gamma)}{N^{m^\gamma/2}} \right) \prod_{i<j} \left( 1+ \frac{P_{ij}^{(2)}}{N} \right) \prod_{k=1}^N \left( 1+ \frac{P_{d, kk}^{(1)}}{\sqrt{N}} \right) = o_\caW(1).
\end{split} \eeq
Note that for $\gamma \in \caG_S$, $Z(\gamma) = \prod_{i<j: m^\gamma_{ij} = 1} P_{ij}^{(1)}$.

To prove the claim \eqref{eq:claim_decomposition}, we first decompose a graph $\gamma \in \caG_L^2$ into $\gamma^{(2)}$ and $\gamma \setminus \gamma^{(2)}$, where $\gamma^{(2)}$ is the $2$-submulticycle of $\gamma$ introduced in the proof of Lemma \ref{lem:self-loop_universality}. In this decomposition, the multicycle $\gamma \setminus \gamma^{(2)}$ consists of self-loops only, and also $Z(\gamma) = Z(\gamma^{(2)}) Z(\gamma \setminus \gamma^{(2)})$. Thus, if we let $\caG_{SL}$ be the set of all multicycles consisting of self-loops only, then
\beq \label{eq:Rademacher_decomposition_1}
	\tilde R_2 = \left( \sum_{\gamma \in \caG_C^2} \frac{1}{N^{m^\gamma/2}} Z(\gamma) \right) \left( \sum_{\gamma' \in \caG_{SL}} \frac{1}{N^{m^{\gamma'}/2}} Z(\gamma') \right) = \left( \sum_{\gamma' \in \caG_{SL}} \frac{1}{N^{m^{\gamma'}/2}} Z(\gamma') \right) R_2.
\eeq
Further, by decomposing $\gamma' \in \caG_{SL}$ into the union of self-loops, we find that
\[
	\sum_{\gamma' \in \caG_{SL}} \frac{1}{N^{m^{\gamma'}/2}} Z(\gamma') = \prod_{k=1}^N \left( 1+ \frac{P_{d, kk}^{(1)}}{\sqrt{N}} + \frac{P_{d, kk}^{(2)}}{N} \right).
\]
In the right side of the equation above, we can easily see that the term involving $P_{d, kk}^{(2)}$ is negligible by naive power counting. Thus,
\beq \label{eq:Rademacher_decomposition_2}
	\sum_{\gamma' \in \caG_{SL}} \frac{1}{N^{m^{\gamma'}/2}} Z(\gamma') = \prod_{k=1}^N \left( 1+ \frac{P_{d, kk}^{(1)}}{\sqrt{N}} \right) + o_\caW(1).
\eeq

We next show that
\beq \label{eq:Rademacher_decomposition_3}
	R_2 = \left( \sum_{\gamma \in \caG_C^1} \frac{1}{N^{m^\gamma/2}} Z(\gamma) \right) \prod_{i<j} \left( 1+ \frac{P_{ij}^{(2)}}{N} \right) + o_\caW(1) = R_1 \prod_{i<j} \left( 1+ \frac{P_{ij}^{(2)}}{N} \right) + o_\caW(1).
\eeq
Notice that the product in the right side of \eqref{eq:Rademacher_decomposition_3},
\[
	\prod_{i<j} \left( 1+ \frac{P_{ij}^{(2)}}{N} \right) = \sum_{\gamma' \in \caG_D} \frac{1}{N^{m^{\gamma'}/2}} Z(\gamma'),
\]
where $\caG_D$ denotes the set of all multicycles consisting of double edges only, without self-loops. For $\gamma \in \caG_C^1$ and $\gamma' \in \caG_D$, we define the sum $\gamma + \gamma'$ to be the $3$-multicycle $\wt \gamma \equiv \wt \gamma(\gamma, \gamma')$ such that $m^{\wt\gamma}_{ij} = m^\gamma_{ij} + m^{\gamma'}_{ij}$. Note that any $3$-multicycle can be uniquely decomposed into the sum of a multicycle in $\caG_C^1$ and a multicycle in $\caG_D$. We also define
\[
	\wt Z(\wt \gamma(\gamma, \gamma')) := \left( \prod_{i<j} P_{ij}^{(m^\gamma_{ij})} P_{ij}^{(m^{\gamma'}_{ij})} \right).
\]
Then,
\[ \begin{split}
	\left( \sum_{\gamma \in \caG_C^1} \frac{1}{N^{m^\gamma/2}} Z(\gamma) \right) \prod_{i<j} \left( 1+ \frac{P_{ij}^{(2)}}{N} \right) = \sum_{\wt \gamma \in \caG_C^3} \frac{1}{N^{m^\gamma/2}} \wt Z(\wt \gamma) =: R_3.
\end{split} \]
Recall the definition of $X_3$ in \eqref{eq:object_3}. If we consider the case where the prior is Rademacher and $P^{(3)}_{ij} = P^{(1)}_{ij} P^{(2)}_{ij}$, then $X_3$ is equal to $R_3$. (Note that we can check that the assumptions on $P^{(3)}_{ij}$ holds in this case since $\E[P^{(3)}_{ij}] = \E[P^{(1)}_{ij} P^{(2)}_{ij}] = 0$ and $\E[P^{(2)}_{ij} P^{(3)}_{ij}] = \E[P^{(1)}_{ij} (P^{(2)}_{ij})^2] = 0$.) Then, from the $3$-multicycle estimate, Lemma \ref{lem:3-cycle_estimate}, we find that \eqref{eq:Rademacher_decomposition_3} holds.

It remains to prove that
\beq \label{eq:Rademacher_decomposition_4}
	R_1 = \prod_{\gamma \in \caG_S} \left( 1+ \frac{Z(\gamma)}{N^{m^\gamma/2}} \right) + o_\caW(1).
\eeq
We can prove \eqref{eq:Rademacher_decomposition_4} by exactly following the first part of the proof of Proposition 2.2 in \cite{AizenmanLebowitzRuelle}. We now find from \eqref{eq:Rademacher_decomposition_1}, \eqref{eq:Rademacher_decomposition_2}, \eqref{eq:Rademacher_decomposition_3}, and \eqref{eq:Rademacher_decomposition_4} that the claim \eqref{eq:claim_decomposition} holds.

With the claim \eqref{eq:claim_decomposition}, we are now ready to finish the proof of Proposition \ref{prop:main}.
\begin{proof}[Proof of Proposition \ref{prop:main}]
Taking logarithm of the product in the left side of \eqref{eq:claim_decomposition},
\[ \begin{split}
	&\log \prod_{\gamma \in \caG_S} \left( 1+ \frac{Z(\gamma)}{N^{m^\gamma/2}} \right) \prod_{i<j} \left( 1+ \frac{P_{ij}^{(2)}}{N} \right) \prod_{k=1}^N \left( 1+ \frac{P_{d, kk}^{(1)}}{\sqrt{N}} \right) \\
	&= \sum_{\gamma \in \caG_S} \log\left( 1+ \frac{Z(\gamma)}{N^{m^\gamma/2}} \right) + \sum_{i<j} \log \left( 1+ \frac{P_{ij}^{(2)}}{N} \right) + \sum_{k=1}^N \log \left( 1+ \frac{P_{d, kk}^{(1)}}{\sqrt{N}} \right) =: L_1 + L_2 + L_3.
\end{split} \]

To prove that $L_1$ converges to a Gaussian, we follow the analysis in the cluster expansion \cite{AizenmanLebowitzRuelle}. Consider
\[
	L_1 = \sum_{\gamma \in \caG_S} \log\left( 1+ \frac{Z(\gamma)}{N^{m^\gamma/2}} \right) = \sum_{\gamma \in \caG_S} \left( \frac{Z(\gamma)}{N^{m^\gamma/2}} - \frac{1}{2} \left(\frac{Z(\gamma)}{N^{m^\gamma/2}} \right)^2 \right) + o_\caW(1). 
\]
Since the number of distinct simple cycles with the length $s$ is $N(N-1) \dots (N-s+1)/(2s)$, by Wick's identity for the moments of the Gaussian distribution, we find
\[
	L_1 \Rightarrow \caN(-\rho_1, 2\rho_1),
\]
where
\[
	\rho_1 = \frac{1}{2} \sum_{s=3}^{\infty} \frac{\F^s}{2s} = -\frac{1}{4} \left( \log(1-\F) + \F + \frac{\F^2}{2} \right).
\]
See Lemma 3.1 in \cite{AizenmanLebowitzRuelle} for more detail.

For $L_2$, since $P_{ij}^{(2)}$'s are i.i.d., we can simply apply the central limit theorem and the law of large numbers to find that
\[
	L_2 = \sum_{i<j} \log \left( 1+ \frac{P_{ij}^{(2)}}{N} \right) = \sum_{i<j} \left( \frac{P_{ij}^{(2)}}{N} - \frac{1}{2} \left( \frac{P_{ij}^{(2)}}{N} \right)^2 \right) + o_\caW(1) \Rightarrow \caN(-\rho_2, 2\rho_2),
\]
where $\rho_2 = \G/4$. Similarly, since $P_{d, kk}^{(1)}$'s are i.i.d., again from the central limit theorem and the law of large numbers,
\[
	L_3 = \sum_{k=1}^N \log \left( 1+ \frac{P_{d, kk}^{(1)}}{\sqrt{N}} \right) = \sum_{k=1}^N \left( \frac{P_{d, kk}^{(1)}}{\sqrt{N}} - \frac{1}{2} \left( \frac{P_{d, kk}^{(1)}}{\sqrt{N}} \right)^2 \right) + o_\caW(1) \Rightarrow \caN(-\rho_3, 2\rho_3),
\]
where $\rho_3 = \cn_1'/2$. 

Since $P_{ij}^{(1)}$ and $P_{ij}^{(2)}$ are orthogonal to each other and $P_{d, kk}^{(1)}$ is independent from $P_{ij}^{(1)}$ and $P_{ij}^{(2)}$, from \eqref{eq:claim_decomposition}, we find that $\log X$ converges in distribution to Gaussian with the mean $(-\rho_1 - \rho_2 - \rho_3)$ and the variance $2(\rho_1 + \rho_2 + \rho_3)$. This completes the proof of Proposition \ref{prop:main}.
\end{proof}

\section{Technical details} \label{sec:proofs}

In this section, we provide several technical details, mostly the proofs of the lemmas used in the previous sections.

\begin{proof}[Proof of Lemma \ref{lem:Gaussian_second_moment}]
We adapt the idea of the proof of Theorem 3.10 in \cite{Perry2018}. To prove the first part of the lemma, we notice that $\sqrt{N} \langle \bsx, \bsx' \rangle$ converges in distribution to $\caN(0, 1)$ by the central limit theorem. Thus, for the event $\Omega_N$ in Assumption \ref{assump:second_moment}, by choosing $\Omega_N' = \Omega_N$
\[
	\indi(\Omega_N' (\bsx)) \indi(\Omega_N' (\bsx')) \exp \left( \frac{N \SNR \langle \bsx, \bsx' \rangle^2}{2} \right) \Rightarrow \exp \left( \frac{\SNR \chi_1^2}{2} \right),
\]
where $\chi_1^2$ is the chi-squared distribution with the degree of freedom $1$. Further, if we let $\delta>0$ be a small constant satisfying $\SNR(1+\delta) < \SNR_c$, then
\[ \begin{split}
	&\E_{\bsx, \bsx' \sim \caX} \left[ \left( \indi(\Omega_N' (\bsx)) \indi(\Omega_N' (\bsx')) \exp \left( \frac{N \SNR \langle \bsx, \bsx' \rangle^2}{2} \right) \right)^{1+\delta} \right] \\
	&\leq \E_{\bsx, \bsx' \sim \caX} \left[ \indi(\Omega_N' (\bsx)) \indi(\Omega_N' (\bsx')) \exp \left( \frac{N \SNR_c \langle \bsx, \bsx' \rangle^2}{2} \right) \right] = O(1).
\end{split} \]
Thus, we have the convergence in $L^{1+\delta}$, which implies the uniform integrability and also, together with the convergence in distribution, implies the convergence in expectation. The limiting value, $(1-\SNR)^{-1/2}$, can be readily computed from the moment generating function of the chi-squared distribution.

To prove the second part of the lemma, we define $\Omega_N'(\bsx) = \Omega_N (\bsx) \cap \{ |\sum_i x_i^2 - 1| \leq N^{-1/3} \}$. Then, $\p(\Omega_N'^c) = o(1)$. Since $\| \bsx \|$ converges to $1$ in probability, we again find that
\[
	\indi(\Omega_N' (\bsx)) \indi(\Omega_N' (\bsx')) \exp \left( \frac{N \SNR \langle \bsx, \bsx' \rangle^2}{2\|\bsx\|^2\|\bsx'\|^2} \right) \Rightarrow \exp \left( \frac{\SNR \chi_1^2}{2} \right).
\]
Let $\delta' > 0$ be a small constant such that $\SNR (1+\delta') / \| \bsx \|^2 \| \bsx' \|^2 < \SNR_c$ on the event $\{ |\sum_i x_i^2 - 1| \leq N^{-1/3} \} \cap \{ |\sum_i (x'_i)^2 - 1| \leq N^{-1/3} \}$, for any sufficiently large $N$. We then have the convergence in $L^{1+\delta}$ as in the proof of the first part of the lemma. We now proceed as in the first part and conclude that the second part of the lemma holds.
\end{proof}

\begin{proof}[Proof of Lemma \ref{lem:ALR_cutoff}]
	We closely follow the proof of Lemma 3.3 in \cite{AizenmanLebowitzRuelle}. Suppose that $\gamma \in \caG_C^1$. We can decompose it into the union $\gamma_1 \cup \gamma_2 \cup \dots \cup \gamma_n$ where each $\gamma_i$ ($i=1, 2, \dots, n$) is a simple cycle defined in Definition \ref{def:simple_cycle}. To see this, we choose a node $i_1$ with the highest degree in $\gamma$ and choose another node $i_2$ adjacent to $i_1$. We repeat it to choose $i_3, i_4, \dots$ with a rule that whenever $i_k$ is adjacent to $i_1$, we choose $i_{k+1} = i_1$; since $\gamma$ is a $1$-multicycle, this is always possible. By construction, the resulting graph with the edges $(i_1, i_2), (i_2, i_3), \dots, (i_k, i_{k+1})$ is a $1$-multicycle, which we call $\gamma'$, and the degree of $i_1$ in $\gamma'$, $d_{i_1}^{\gamma'}$ is $2$. If the degree $d_{i_1}^\gamma$ was $2$, then $\gamma'$ must be simple and we let $\gamma_1 = \gamma'$. Otherwise, we repeat the same procedure with $\gamma'$. Since $d_{i_1}^{\gamma'} < d_{i_1}^{\gamma}$, we have $n^{\gamma'} < n^{\gamma}$ and the procedure will be finished. After decomposing $\gamma'$, we repeat the same procedure with $\gamma\setminus \gamma'$, and we get the decomposition we desired.

Since $n^{\gamma} = \sum_{i=1}^n n^{\gamma_i}$ with the decomposition, we have the inequality
\beq \label{eq:cutoff_a_priori}
	\sum_{\gamma \in \caG_C^1} \left( \frac{\alpha}{N} \right)^{n^{\gamma}} \leq \prod_{\gamma' \in \caG_S} \left(1 + \left( \frac{\alpha}{N} \right)^{n^{\gamma'}} \right).
\eeq
(The inequality \eqref{eq:cutoff_a_priori} is strict, since the right side contains terms of a form
\[
	\left( \frac{\alpha}{N} \right)^{n^{\gamma'_1}} \left( \frac{\alpha}{N} \right)^{n^{\gamma'_2}} \dots \left( \frac{\alpha}{N} \right)^{n^{\gamma'_k}}
\]
where some of the $\gamma'_i$'s share a common edge.) Note that the inequality \eqref{eq:cutoff_a_priori} holds even when $\alpha \geq 1$.

Now, consider a function
\[
	\Phi(k) =
	\begin{cases}
		e^{\delta s} & \text{if } k \geq s \\
		0 & \text{if } k < s
	\end{cases},
\]
with
\[
	\delta = \log \left[ (\alpha)^{-1} \left( 1 - \frac{\log s}{2s} \right) \right],
\]
or equivalently $e^{\delta} \alpha = 1 - (\log s)/(2s)$. Then, since $\Phi(k) \leq e^{\delta k}$, 
\[
	e^{\delta s} \sum_{\gamma \in \caG_C^1, n^\gamma \geq s} \left( \frac{\alpha}{N} \right)^{n^{\gamma}} = \sum_{\gamma \in \caG_C^1} \Phi(n^{\gamma}) \left( \frac{\alpha}{N} \right)^{n^{\gamma}} \leq \sum_{\gamma \in \caG_C^1} \left( \frac{e^{\delta} \alpha}{N} \right)^{n^{\gamma}}.
\]
Thus, from \eqref{eq:cutoff_a_priori},
\[
	e^{\delta s} \sum_{\gamma \in \caG_C^1, n^\gamma \geq s} \left( \frac{\alpha}{N} \right)^{n^{\gamma}} \leq \prod_{\gamma' \in \caG_S} \left(1 + \left( \frac{e^{\delta} \alpha}{N} \right)^{n^{\gamma'}} \right) \leq \exp \left( \sum_{\gamma' \in \caG_S} \left( \frac{e^{\delta} \alpha}{N} \right)^{n^{\gamma'}} \right).
\]
Since the number of distinct simple cycles with length $k$ is $N(N-1) \dots (N-k+1) / (2k)$, for $k \geq 3$, we obtain
\[ \begin{split}
	&\sum_{\gamma \in \caG_C^1, n^\gamma \geq s} \left( \frac{\alpha}{N} \right)^{n^{\gamma}} \leq e^{-\delta s} \exp \left( \sum_{k=3}^{\infty} \left( \frac{(e^\delta \alpha)^k}{2k} \right) \right) \\
	&= \exp \left( -\delta s + \frac{1}{2} \left( -\log (1-e^\delta \alpha) -e^\delta \alpha - \frac{(e^\delta \alpha)^2}{2} \right) \right) \\
	&= \exp \left( s( \log \alpha) - s \log \left( 1 - \frac{\log s}{2s} \right) -\frac{1}{2} \log \left( \frac{\log s}{2s} \right) - \left( 1 - \frac{\log s}{2s} \right) - \frac{1}{2} \left( 1 - \frac{\log s}{2s} \right)^2 \right) \\
	&= \exp \left( s(\log \alpha) + \log s - \log \log s + O(1) \right).
\end{split} \]
This proves the desired lemma.
\end{proof}

\begin{proof}[Proof of Lemma \ref{lem:Y_gamma_dichotomy}]
We separate the proof into two cases.

\textit{Case 1.} Suppose that $x_i = \frac{1}{\sqrt{N}} v_i y_i$. If all nodes of $\gamma$ are with degree $2$, 
\[
	Y(\gamma) = \E_\caY \left[ \indi(\Omega) \prod_{r=1}^{r_0} y_{i_r}^2 \right] \leq \E_\caY \left[ \prod_{r=1}^{r_0} y_{i_r}^2 \right] = 1
\]
and $1- Y(\gamma)= o(1)$ since $\E[\indi(\Omega^c)] = o(1)$. This proves the first part of the lemma. 

To prove the second part, assume that $\gamma$ has the nodes $i_1, \dots, i_r$ and the degree of the one of the node in $\gamma$ is $4$ or higher. Note that $\E_\caY[y_{i_\ell}^{d_{i_\ell}}] \geq 1$ for $\ell = 1, 2, \dots, r$, and since the prior is not Rademacher, one of them is larger than $1 +\delta$ for some $\delta > 0$. Then,
\[
	Y(\gamma) \geq (1+o(1)) \prod_{\ell=1}^r \E_\caY \left[y_{i_\ell}^{d_{i_\ell}^\gamma} \right] > 1+\frac{\delta}{2}
\]
if $N$ is sufficiently large. This proves the second part of the lemma.

\textit{Case 2.} Suppose that $\frac{x_i}{\| \bsx \|} = \frac{1}{\sqrt{N}} v_i y_i$.

We first prove the second part of the lemma in this case. Recall the definition of $\Omega$ in Definition \ref{defn:high-probability}. If $\gamma$ has the nodes $i_1, \dots, i_r$, then
\[
	Y(\gamma) = \E_\caX \left[ N^{m^\gamma} \indi(\Omega) \prod_{\ell=1}^r \left( x_{i_\ell}^{d_{i_\ell}^\gamma} / \| \bsx \|^{d_{i_\ell}^\gamma} \right) \right] \geq N^{m^\gamma} (1+o(1)) \prod_{\ell=1}^r \E_\caX \left[x_{i_\ell}^{d_{i_\ell}^\gamma} \right]
\]
Since $x_i$'s are i.i.d., we can now proceed as in the proof for \textit{Case 1.} to prove the second part of the lemma for \textit{Case 2.}
 
To prove the first part, we use the induction on the number of nodes in $\gamma$. When $\gamma$ has only one node, since $\sum_{i=1}^{N} y_i^2 = N$ and $\{y_i\}$ are exchangeable,
\[
	Y(\gamma) = \frac{1}{N} \E_\caY \left[ \indi(\Omega) \sum_{k=1}^{N} y_k^2 \right] = 1.
\]
Suppose that $Y(\gamma) \leq 1$ and $1-Y(\gamma) = o(1)$ for any $\gamma$ with $r$ nodes or less, whose all nodes are with degree $2$. For a graph with $(r+1)$ nodes (with degree $2$), assume without loss of generality that the nodes are $1, 2, \dots, (r+1)$. Then, since $\{y_i\}$ are exchangeable,
\[ \begin{split}
	&N \E_\caY \left[ \indi(\Omega) y_1^2 y_2^2 \dots y_r^2 \right] = \E_\caY \left[ \indi(\Omega) y_1^2 y_2^2 \dots y_r^2 \left( \sum_{i=1}^N y_i^2 \right) \right] \\
	&= r \E_\caY \left[ \indi(\Omega) y_1^4 y_2^2 \dots y_r^2 \right] + (N-r) \E_\caY \left[ \indi(\Omega) y_1^2 y_2^2 \dots y_r^2 y_{r+1}^2 \right].
\end{split} \]
The first term in the right side is bounded, since
\[
	\E_\caY \left[ \indi(\Omega) y_1^4 y_2^2 \dots y_r^2 \right] \leq (1+o(1)) N^{r+1} \E_\caX \left[ x_1^4 \right] \prod_{\ell=2}^r \E_\caY \left[ x_\ell^2 \right] = O(1).
\]
We now find that $Y(\gamma) \leq 1$ and $1-Y(\gamma) = o(1)$ for this case, since
\[ \begin{split}
	&\E_\caY \left[ \indi(\Omega) y_1^2 y_2^2 \dots y_r^2 y_{r+1}^2 \right] \\
	&= \E_\caY \left[ \indi(\Omega) y_1^2 y_2^2 \dots y_r^2 \right] - \frac{r}{N-r} \left( \E_\caY \left[ \indi(\Omega) y_1^4 y_2^2 \dots y_r^2 \right] - \E_\caY \left[ \indi(\Omega) y_1^2 y_2^2 \dots y_r^2 \right] \right) 
\end{split} \]
From the first part of the lemma and the induction hypothesis,
\[
	\E_\caY \left[ \indi(\Omega) y_1^2 y_2^2 \dots y_r^2 \right] \leq 1 < \E_\caY \left[ \indi(\Omega) y_1^4 y_2^2 \dots y_r^2 \right].
\]
This proves the first part for a graph with $(r+1)$ nodes and completes the proof of Lemma \ref{lem:Y_gamma_dichotomy}.
\end{proof}

\begin{proof}[Proof of Lemma \ref{lem:Rademacher_difference_positive}]
Assume the GOE disorder for the log likelihood ratio in Theorem \ref{thm:log_LR}, i.e.,
\[
	p^{(n)}_{ij} = \frac{(-1)^n \SNR^{n/2} p^{(n)}(\sqrt{N} W_{ij})}{n! p(\sqrt{N} W_{ij})} \quad (n=1, 2, 3, 4), \qquad p^{(n)}_{d, kk} = \frac{(-1)^n \SNR^{n/2} p_d^{(n)}(\sqrt{N} W_{kk})}{n! p_d(\sqrt{N} W_{kk})} \quad (n=1, 2).
\]
where $W_{ij}$ and $W_{kk}$ are independent Gaussian random variables with variances $N^{-1}$ and $2N^{-1}$, respectively. 

Recall that the left side of \eqref{eq:Gaussian_second_moment_1} or \eqref{eq:Gaussian_second_moment_2} is the second moment of the log likelihood ratio with GOE disorder, conditioned on $\Omega_N'$. (See Section \ref{subsec:threshold}.) Further, from the proof of Lemma \ref{lem:Gaussian_second_moment}, it is not hard to see that the conclusion of Lemma \ref{lem:Gaussian_second_moment} does not change even if $\Omega_N'$ is replaced by $\Omega$. Thus, from Lemma \ref{lem:Gaussian_second_moment} and the proof of Theorem \ref{thm:log_LR},
\beq \label{eq:Gaussian_second_moment}
	\E\left[\left(\sum_{\gamma \in \caG_L^4} \frac{1}{N^{m^\gamma/2}} Y(\gamma) Z(\gamma)\right)^2\right] = (1-\SNR)^{-1/2} + o(1).
\eeq
From the GOE assumption, $p_{ij}^{(n)}$ is a constant multiple of the $n$-th Hermite polynomial with respect to the Gaussian of the form $e^{-x^2/2}$ and $p_{d, kk}^{(n)}$ is a constant multiple of the $n$-th Hermite polynomial with respect to the Gaussian of the form $e^{-x^2/4}$. Due to the orthogonality of the Hermite polynomials, we can easily see that if $\gamma$ and $\gamma'$ are distinct multicycles, then they are orthogonal. Thus, we find from \eqref{eq:Gaussian_second_moment} that
\[
	\sum_{\gamma \in \caG_L^4} \frac{1}{N^{m^\gamma}} Y(\gamma)^2 \E \left[ Z(\gamma)^2 \right] = (1-\SNR)^{-1/2} + o(1).
\]
In particular, considering the Rademacher prior,
\beq \label{eq:Rademacher_limit}
	\sum_{\gamma \in \caG_L^4} \frac{1}{N^{m^\gamma}} \E \left[ Z(\gamma)^2 \right] = (1-\SNR)^{-1/2} + o(1).
\eeq
Comparing the equations above, we obtain
\beq \label{eq:Rademacher_difference}
	\sum_{\gamma \in \caG_L^4} \frac{1}{N^{m^\gamma}} \left( Y(\gamma)^2 -1 \right) \E \left[ Z(\gamma)^2 \right] = o(1).
\eeq

To proceed further, we introduce the $4$-multicycle cutoff property for Rademacher prior, which is the following lemma.
\begin{lem}[4-multicycle cutoff property for Rademacher prior with GOE disorder] \label{lem:4-cycle_Rademacher_cutoff}
	Assume the GOE disorder for the log likelihood ratio in Theorem \ref{thm:log_LR}. If $\SNR<\SNR_c$, then for any fixed $\epsilon>0$, there exists ($N$-independent) constant $K_\epsilon''' \equiv K_\epsilon'''(\SNR)$ such that for any $s>K_\epsilon'''$, and for any sufficiently large $N$,
\[ 
	\sum_{\gamma\in \caG_L^4, n^\gamma \geq s} \frac{1}{N^{m^\gamma}} \E[Z(\gamma)^2] < \epsilon.
\]
\end{lem}
We will prove Lemma \ref{lem:4-cycle_Rademacher_cutoff} later in this section. With Lemma \ref{lem:4-cycle_Rademacher_cutoff}, we can estimate the sum in \eqref{eq:Rademacher_difference} for the case $Y(\gamma) \leq 1$ and $n^\gamma \geq K_\epsilon'''$ by
\[ \begin{split}
	&\sum_{\gamma\in \caG_L^4, n^\gamma \geq K_\epsilon''', Y(\gamma) \leq 1} \frac{1}{N^{m^\gamma}} \left( 1 -Y(\gamma)^2 \right) \E \left[ Z(\gamma)^2 \right] \leq \sum_{\gamma\in \caG_L^4, n^\gamma \geq K_\epsilon''', Y(\gamma) \leq 1} \frac{1}{N^{m^\gamma}} \E \left[ Z(\gamma)^2 \right] \\
	&\leq \sum_{\gamma\in \caG_L^4, n^\gamma \geq K_\epsilon'''} \frac{1}{N^{m^\gamma}} \E \left[ Z(\gamma)^2 \right] < \epsilon.
\end{split} \]
If $Y(\gamma) \leq 1$ and $n^\gamma < K_\epsilon'''$, we find from Lemma \ref{lem:Y_gamma_dichotomy} that all nodes of $\gamma$ are with degree $2$. Since $1-Y(\gamma) = o(1)$ in this case, by following the argument we used to derive \eqref{eq:normalized_1-cycle_first_term_2}, we find that
\[
	\sum_{\gamma\in \caG_L^4, n^\gamma \leq K_\epsilon''', Y(\gamma) \leq 1} \frac{1}{N^{m^\gamma}} \left( 1 -Y(\gamma)^2 \right) \E \left[ Z(\gamma)^2 \right] < \epsilon.
\]
Thus, we find that
\[
	\sum_{\gamma\in \caG_L^4, Y(\gamma) \leq 1} \frac{1}{N^{m^\gamma}} \left( 1 -Y(\gamma)^2 \right) \E \left[ Z(\gamma)^2 \right] = o(1).
\]
Plugging it into \eqref{eq:Rademacher_difference}, we also find that
\[
	\sum_{\gamma\in \caG_L^4, Y(\gamma) > 1} \frac{1}{N^{m^\gamma}} \left( Y(\gamma)^2 -1 \right) \E \left[ Z(\gamma)^2 \right] = o(1).
\]
Since the summand is positive, by restricting $\gamma$ to be in $\caG_C^1$, we find that the desired lemma holds.
\end{proof}

\begin{proof}[Proof of Lemma \ref{lem:cutoff_general}]
We again separate the sum into the case $Y(\gamma) \leq 1$ and the case $Y(\gamma) > 1$, which leads us to
\[ 
	\sum_{\gamma\in \caG_C^1, n^\gamma \geq s} \left( \frac{\F}{N} \right)^{n^{\gamma}} Y(\gamma)^2 
	= \sum_{\gamma\in \caG_C^1, Y(\gamma)\leq 1, n^\gamma \geq s} \left( \frac{\F}{N} \right)^{n^{\gamma}} Y(\gamma)^2 + \sum_{\gamma\in \caG_C^1, Y(\gamma) > 1, n^\gamma\geq s} \left( \frac{\F}{N} \right)^{n^{\gamma}} Y(\gamma)^2. 
\]
For the first term in the right side,	from Lemma \ref{lem:ALR_cutoff}, for given $\epsilon > 0$, there exists $K_\epsilon$ such that
\[ 
	\sum_{\gamma\in \caG_C^1, n^\gamma\ge K_\epsilon} \left( \frac{\F}{N} \right)^{n^{\gamma}} < \epsilon.
\]
Then,
\beq \label{eq:cutoff_general_target1}
	\sum_{\gamma\in \caG_C^1, Y(\gamma)\leq 1, n^\gamma \geq s} \left( \frac{\F}{N} \right)^{n^{\gamma}} Y(\gamma)^2 \leq \sum_{\gamma\in \caG_C^1, Y(\gamma)\leq 1, n^\gamma \geq s} \left( \frac{\F}{N} \right)^{n^{\gamma}} < \epsilon.
\eeq

For the case $Y(\gamma) > 1$, from Lemma \ref{lem:Rademacher_difference_positive}, for given $\epsilon > 0$, there exists $K_\epsilon'$ such that
\[
	\sum_{\gamma\in \caG_C^1, Y(\gamma)>1, n^\gamma \geq K'_\epsilon} \left( \frac{\F}{N} \right)^{n^{\gamma}} (Y(\gamma)^2-1)< \epsilon.
\]
Now, choosing $s>\max\{K_\epsilon, K'_\epsilon\}$, we find that
\[ \begin{split}
		&\sum_{\gamma\in \caG_C^1, Y(\gamma)>1, n^\gamma \geq s} \left( \frac{\F}{N} \right)^{n^{\gamma}} Y(\gamma)^2 \\ 
		&\leq \sum_{\gamma\in \caG_C^1, Y(\gamma)>1, n^\gamma \geq K_\epsilon'} \left( \frac{\F}{N} \right)^{n^{\gamma}} (Y(\gamma)^2-1)+ \sum_{\gamma\in \caG_C^1, Y(\gamma)>1, n^\gamma \geq K_\epsilon} \left( \frac{\F}{N} \right)^{n^{\gamma}} < 2\epsilon.
\end{split} \]
Combining it with \eqref{eq:cutoff_general_target1}, we prove the desired lemma.
\end{proof}

\begin{proof}[Proof of Lemma \ref{lem:4-cycle_Rademacher_cutoff}]
	We closely follow the proof of Lemma \ref{lem:4-cycle_estimate}.
	Since
	\begin{equation}\label{eq:Rademacher_cutoff_GOE}
		\sum_{\gamma\in\caG_L^4, n^{\gamma} \ge s}\frac{1}{N^{m^\gamma}}\E[Z(\gamma)^2] \le \sum_{\gamma\in\caG_L^2, n^{\gamma}\ge s} \frac{1}{N^{m^\gamma}}\E[Z(\gamma)^2]+\sum_{\gamma\in\caG_L^4\setminus\caG_L^2}\frac{1}{N^{m^\gamma}}\E[Z(\gamma)^2].
	\end{equation}
	For the second term, we group the multicycles in $\caG_L^4\setminus\caG_L^2$ according to their equivalence class defined through the similarity in Definition \ref{def:similar}. Then we get
\begin{align*}
	 &\sum_{\gamma\in \caG_L^4\setminus \caG_L^2}\frac{1}{N^{m^\gamma}}\E[Z(\gamma)^2]=\sum_{\gamma\in\caG_L^2}\E\left[\left(\sum_{\gamma'\in\caG_L^4\setminus\caG_L^2, \gamma'\sim\gamma}\frac{1}{N^{m^{\gamma'} / 2}}Z(\gamma') \right)^2 \right] \\
	 &\le \sum_{\gamma'\in \caG_L^2}\left( \frac{\F}{N} \right)^{n_1^{\gamma'}-\tilde n_1^{\gamma'}}\left( \frac{\G}{N^2} \right)^{n_2^{\gamma'}-\tilde n_2^{\gamma'}}\left( \frac{\F'}{N} \right)^{\tilde n_1^{\gamma'}}\left( \frac{\G'}{N^2} \right)^{\tilde n_2^{\gamma'}} \left[\left(1+\frac{\sqrt{\cn_3}}{N\sqrt{\F}} \right)^{n_1^{\gamma'}} \left( 1+\frac{\sqrt{\cn_4}}{N\sqrt{\G}} \right)^{n_2^{\gamma'}}-1 \right]^2 \\
	 &\le \frac{C}{N}\sum_{\gamma'\in \caG_L^2}\left( \frac{\F}{N} \right)^{n_1^{\gamma'}-\tilde n_1^{\gamma'}}\left( \frac{\G}{N^2} \right)^{n_2^{\gamma'}-\tilde n_2^{\gamma'}}\left( \frac{\F'}{N} \right)^{\tilde n_1^{\gamma'}}\left( \frac{\G'}{N^2} \right)^{\tilde n_2^{\gamma'}},
\end{align*}
for some constant $C$. Since the sum is bounded by $(1-\SNR)^{-\frac{1}{2}}+o(1)$ from \eqref{eq:Rademacher_difference}, we find that the second term of \eqref{eq:Rademacher_cutoff_GOE} is $o(1)$.	

Now for the first term, since $n_\gamma = (n_1^\gamma-\tilde n_1^\gamma)+(n_2^\gamma-\tilde n_2^\gamma)+\tilde n_1^\gamma+\tilde n_2^\gamma$ for $\gamma\in \caG_L^2$, we have
\begin{equation}\label{eq:2-cycle_Rademacher_division}
 \begin{split} 
	\sum_{\gamma\in\caG_L^2, n^\gamma >4s} \frac{1}{N^{m^\gamma}} \E[Z(\gamma)^2] 
	&\leq \sum_{\gamma\in\caG_L^2, n_1^\gamma-\tilde n_1^\gamma >s} \frac{1}{N^{m^\gamma}} \E[Z(\gamma)^2] + \sum_{\gamma\in\caG_L^2, n_2^\gamma-\tilde n_2^\gamma >s} \frac{1}{N^{m^\gamma}} \E[Z(\gamma)^2]  \\																																								
	&\quad + \sum_{\gamma\in\caG_L^2, \tilde n_1^\gamma >s} \frac{1}{N^{m^\gamma}} \E[Z(\gamma)^2] + \sum_{\gamma\in\caG_L^2, \tilde n_2^\gamma >s}\frac{1}{N^{m^\gamma}} \E[Z(\gamma)^2].
 \end{split}
\end{equation}
Our goal is to show that the four sums in the right side of \eqref{eq:2-cycle_Rademacher_division} are small. We have
\[
	\frac{1}{N^{m^\gamma}} \E[Z(\gamma)^2]= \left( \frac{\F}{N} \right)^{n_1^\gamma-\tilde n_1^\gamma}\left( \frac{\G}{N^2}\right)^{n_2^\gamma-\tilde n_2^\gamma} \left( \frac{\cn_1'}{N} \right)^{\tilde n_1^\gamma} \left( \frac{\cn_2'}{N^2} \right)^{\tilde n_2^\gamma}.
\]
First, we consider cutoff for double self-loops, the estimate on the last term in the right side of \eqref{eq:2-cycle_Rademacher_division}. We group the graphs in $\caG_L^2$ according to their submulticycles with 2-multicycles and single self-loops only. Then,
\[
	\sum_{\gamma\in \caG_L^2, \tilde n_2^\gamma >s} \frac{1}{N^{m^\gamma}} \E[Z(\gamma)^2] \leq \left(\sum_{\gamma'\in\caG_L^2, \tilde n_2^{\gamma'} = 0} \frac{1}{N^{m^\gamma}} \E[Z(\gamma)^2] \right) \left(\sum_{r=s}^{\infty}\frac{N^r}{r!}\left(\frac{\cn_2'}{N^2} \right)^r\right).
\]
The first term in the right side is bounded by $(1-\SNR)^{-\frac{1}{2}}+o(1) = O(1)$, and the second part is a tail part of the Taylor series of the $e^{\cn_2' / N}$. Thus, we can find an $N$-independent constant $s_1$ such that if $s > s_1$, the left-hand side is smaller than $\epsilon$. Similarly, the cutoff for the single self-loops, the third term in the right side of \eqref{eq:2-cycle_Rademacher_division}, is given by
\[ 
	\sum_{\gamma\in \caG_L^2, \tilde n_1^\gamma >s} \frac{1}{N^{m^\gamma}} \E[Z(\gamma)^2] \leq \left(\sum_{\gamma'\in\caG_L^2, \tilde n_1^{\gamma'} = 0} \frac{1}{N^{m^\gamma}} \E[Z(\gamma)^2] \right) \left(\sum_{r=s}^{\infty}\frac{N^r}{r!} \left(\frac{\F'}{N} \right)^r\right) 
\]
The second part is a tail part of the Taylor series of the $e^{\F'}$, and thus we can find an $N$-independent constant $s_2$ such that if $s>s_2$, the left-hand side is smaller than $\epsilon$. Similarly, the cutoff for the double edges, the second term in the right side of \eqref{eq:2-cycle_Rademacher_division}, we consider
\[ 
	\sum_{\gamma\in\caG_L^2, n_2^\gamma - \tilde n_2^\gamma > s} \frac{1}{N^{m^\gamma}} \E[Z(\gamma)^2] \leq \left(\sum_{\gamma'\in\caG_L^2, n_2^{\gamma'}-\tilde n_2^{\gamma'} = 0} \frac{1}{N^{m^\gamma}} \E[Z(\gamma)^2] \right) \left( \sum_{r=s}^{\infty}\frac{N^{2r}}{2^r r!}\left(\frac{\G}{N^2}\right)^r \right),
\]
and find an $N$-independent constant $s_3$ such that if $s > s_3$, the left-hand side is smaller than $\epsilon$. Lastly, the cutoff for the single edges, the first term in the right side of \eqref{eq:2-cycle_Rademacher_division},
\[ 
	\sum_{\gamma\in\caG_L^2, n_1^\gamma - \tilde n_1^\gamma > s} \frac{1}{N^{m^\gamma}} \E[Z(\gamma)^2] \leq \left(\sum_{\gamma'\in\caG_L^2, n_1^{\gamma'}-\tilde n_1^{\gamma'} = 0}\frac{1}{N^{m^\gamma}} \E[Z(\gamma)^2] \right) \left( \sum_{\gamma''\in\caG_C^1, n^{\gamma''} > s}\left( \frac{\F}{N} \right)^{m^{\gamma''}} \right).
\]
By Lemma \ref{lem:ALR_cutoff}, we can find an $N$-independent constant $s_4$ such that if $s>s_4$, the left-hand side is smaller than $\epsilon$. Now we can finish the proof of the desired lemma by letting $K_{\epsilon}''' = 4\max \{ s_1, s_2, s_3, s_4 \}$.
\end{proof}

\subsection*{Acknowledgement}
This work was partially supported by National Research Foundation of Korea under grant number RS-2023-NR076695.


\end{document}